\documentclass{amsproc}
\usepackage[margin=1.2in,nomarginpar]{geometry}
\usepackage{amssymb}
\usepackage{amsmath}
\usepackage{mathdots}
\usepackage{amsbsy}
\usepackage{amscd}
\usepackage{amsthm}

\usepackage{url}
\usepackage{xcolor}
\usepackage{amsmath,amssymb,amsthm,amsfonts,longtable}
\usepackage[english]{babel}
\usepackage{tikz-cd}
\usetikzlibrary{cd}
\usetikzlibrary{decorations.markings}
\tikzset{negated/.style={
		decoration={markings,
			mark= at position 0.5 with {
				\node[transform shape] (tempnode) {$\times$};
			}
		},
		postaction={decorate}
	}
}
\usepackage{array}
\usepackage[colorlinks,citecolor=red,urlcolor=blue,bookmarks=false,hypertexnames=true]{hyperref} 
\usepackage{multirow}
\usepackage{pbox}
\newtheorem{theorem}{Theorem}
\newtheorem{corollary}[theorem]{Corollary}
\newtheorem{lemma}[theorem]{Lemma}
\newtheorem{proposition}[theorem]{Proposition}

\newtheorem{remark}[theorem]{Remark}

\newtheorem{example}[theorem]{Example}

\newcommand{\Irr}{\textnormal{Irr}}
\newcommand{\FIrr}{\textnormal{FIrr}}
\newcommand{\cd}{\textnormal{cd}}
\newcommand{\nl}{\textnormal{nl}}
\newcommand{\lin}{\textnormal{lin}}

\newcommand{\gal}{\textnormal{Gal}}

\title[]{A Combinatorial Formula for the Wedderburn Decomposition of Rational Group Algebras and the Rational Representations of Ordinary Metacyclic $p$-groups}
\author{Ram Karan Choudhary}
\address{Indian Institute of Technology, Bhubaneswar, Arugul Campus, Jatni, Khurda-752050, India.}
\email{ramkchoudhary1997@gmail.com}
\author{Sunil Kumar Prajapati$^*$}
\address{Indian Institute of Technology, Bhubaneswar, Arugul Campus, Jatni, Khurda-752050, India.}
\email{skprajapati@iitbbs.ac.in}
\thanks{$^{\textbf{*}}$ Corresponding author.
}
\begin{document}
	\maketitle
	
	\begin{abstract}
		In this article, we present a combinatorial formula for computing the Wedderburn decomposition of the rational group algebra associated with an ordinary metacyclic $p$-group $G$, where $p$ is any prime. We also provide a formula for counting irreducible rational representations of $G$ with distinct degrees and derive a method to explicitly obtain all inequivalent irreducible rational matrix representations of $G$.
	\end{abstract}
	
	\section{Introduction} 
	Finite metacyclic groups were classified by Hempel \cite{Hempel}, while earlier, the classification of finite metacyclic $p$-groups was addressed by several authors in \cite{King,  Liedahl, Liedahl 2, Sim}. In this article, we adopt the classification of metacyclic $p$-groups as given by King \cite{King}. Let $p$ be a prime number, and let $G$ denote a finite, non-cyclic, metacyclic $p$-group. King \cite[Theorem 3.2]{King} demonstrated that each metacyclic $p$-group admits a uniquely reduced presentation, up to isomorphism. Consequently, $G$ can be described by a uniquely reduced presentation, which falls into one of the following two categories:
	
	\begin{enumerate}
		\item[(A)] Ordinary metacyclic groups:
		\begin{equation}\label{prest:metacyclic}
			G = \langle a, b ~ : ~ a^{p^n} = 1, b^{p^m} = a^{p^{n-r}}, bab^{-1} = a^{1+p^{n-s}} \rangle
		\end{equation}
		for certain integers $m$, $n$, $r$ and $s$, where either $p \geq 3$, or $p = 2$ and $s < n - 1$ when $n \geq 2$. The presentation in \eqref{prest:metacyclic} is uniquely reduced if and only if the following conditions are satisfied:
		\begin{enumerate}
			\item Split case: $0 = r \leq s < \min\{m+1, n\}$.
			\item Non-split case: $\max\{1, n-m+1\} \leq r < \min\{s, n-s+1\}$.
		\end{enumerate}
		
		\item[(B)] Exceptional metacyclic groups:
		\begin{equation}\label{prest:exmetacyclic}
			G = \langle a, b ~ : ~ a^{2^n} = 1, b^{2^m} = a^{2^{n-r}}, bab^{-1} = a^{-1+2^{n-s}} \rangle
		\end{equation}
		for certain integers $m$, $n$, $r$ and $s$. The presentation in \eqref{prest:exmetacyclic} is uniquely reduced if and only if the following conditions are satisfied:
		\begin{enumerate}
			\item Split case: $0 = r \leq s < \min\{m+1, n-1\}$.
			\item Non-split case: Either $r = 1$ with $\max\{1, n-m+1\} \leq s < \min\{m, n-1\}$, or $r = 1$, $s = 0$ and $m = 1 < n$ (the later case gives the presentation of a generalized quaternion group).
		\end{enumerate}
	\end{enumerate}
	It is worth noting that ordinary metacyclic $p$-groups encompass all metacyclic $p$-groups except for a class of metacyclic $2$-groups, known as exceptional metacyclic groups. In this work, we focus on the Wedderburn decomposition of the rational group algebra associated with an ordinary metacyclic $p$-group.

	The Artin-Wedderburn theorem ensures that a ring is semisimple if and only if it is isomorphic to a direct sum of matrix rings over division rings. Furthermore, by the Brauer-Witt Theorem (see \cite{Yam}), the Wedderburn components of a rational group algebra are Brauer equivalent to cyclotomic algebras. The study of the Wedderburn decomposition of rational group algebras has garnered significant attention in recent research, primarily due to its implications for understanding various algebraic structures. For instance, the Wedderburn decomposition for a finite group $G$ can help to describe the automorphism group of the rational group algebra $\mathbb{Q}G$ (see \cite{Herman, Olivieri}) and contributes to the analysis of the unit group of the integral group ring $\mathbb{Z}G$ (see \cite{Rio-Rui, Jes-Rio, Jes-Lea, Rit-Seh}). The Wedderburn decomposition of the rational group algebras for various families of groups has been studied extensively in \cite{ BG, BG1, BM14, Ram, Jes-Lea-Paq, ODRS04, Olt07, PW}, where techniques such as field of character values, Shoda pairs, numerical representations of cyclotomic algebras and so forth, have been used to compute the simple components of rational group algebras.

	In our previous work \cite{Ram2}, we developed a combinatorial formula for computing the Wedderburn decomposition of the rational group algebra of a split metacyclic $p$-group, where $p$ is an odd prime. In this article, we extend that result to cover all ordinary metacyclic $p$-groups, for any prime $p$. For an ordinary metacyclic $p$-group $G$ as defined in \eqref{prest:metacyclic}, we prove Theorem \ref{thm:wedderburnmetacyclic}, which offers a combinatorial method to determine the Wedderburn decomposition of $\mathbb{Q}G$. Our result simplifies the computation of the Wedderburn decomposition of $\mathbb{Q}G$ using only the numerical values of $n$, $m$, $r$ and $s$.
	
	\begin{theorem}\label{thm:wedderburnmetacyclic}
		Let $p$ be a prime and $\zeta_d$ a primitive $d$-th root of unity for some positive integer $d$. Consider an ordinary, finite, non-cyclic metacyclic $p$-group $G$, with a unique reduced presentation:
		\begin{equation*}
			G = \langle a, b ~ : ~ a^{p^n} = 1, \, b^{p^m} = a^{p^{n-r}}, \, bab^{-1} = a^{1+p^{n-s}} \rangle,
		\end{equation*}
		for certain integers $m, n, r$ and $s$ (as defined in \eqref{prest:metacyclic}). Then we have the following.
		\begin{enumerate}
			\item {\bf Case ($n-s \geq m$).} In this case,
			$$\mathbb{Q}G \cong \mathbb{Q} \bigoplus_{\lambda=1}^m (p^\lambda+p^{\lambda-1})\mathbb{Q}(\zeta_{p^\lambda}) \bigoplus_{\lambda=m+1}^{n-s}p^m \mathbb{Q}(\zeta_{p^\lambda}) \bigoplus_{t=1}^{s} p^{m-t}M_{p^t}(\mathbb{Q}(\zeta_{p^{n-s}})).$$
			\item {\bf Case ($n-s < m$).} Suppose $m = (n-s)+k$. Then we have following two sub-cases.
			\begin{enumerate}
				\item {\bf Sub-case ($k \leq s-r$).} In this sub-case, 
				\begin{align*}
					\mathbb{Q}G \cong & \mathbb{Q} \bigoplus_{\lambda=1}^{n-s} (p^\lambda+p^{\lambda-1})\mathbb{Q}(\zeta_{p^\lambda}) \bigoplus_{\lambda=n-s+1}^{m}p^{n-s} \mathbb{Q}(\zeta_{p^\lambda}) \bigoplus_{t=1}^{k-1} p^{n-s}M_{p^t}(\mathbb{Q}(\zeta_{p^{n-s}}))\\ &\bigoplus_{t=1}^{k-1}\bigoplus_{\lambda=n-s+1}^{m-t} (p^{n-s}-p^{n-s-1})M_{p^t}(\mathbb{Q}(\zeta_{p^{\lambda}})) \bigoplus_{t=k}^{s} p^{m-t}M_{p^t}(\mathbb{Q}(\zeta_{p^{n-s}})).
				\end{align*}
				\item {\bf Sub-case ($k > s-r$).} In this sub-case, 
				\begin{align*}
					\mathbb{Q}G \cong & \mathbb{Q} \bigoplus_{\lambda=1}^{n-s} (p^\lambda+p^{\lambda-1})\mathbb{Q}(\zeta_{p^\lambda}) \bigoplus_{\lambda=n-s+1}^{m}p^{n-s} \mathbb{Q}(\zeta_{p^\lambda}) \bigoplus_{t=1}^{s-r} p^{n-s}M_{p^t}(\mathbb{Q}(\zeta_{p^{n-s}}))\\ &\bigoplus_{t=1}^{s-r}\bigoplus_{\lambda=n-s+1}^{m-t} (p^{n-s}-p^{n-s-1})M_{p^t}(\mathbb{Q}(\zeta_{p^{\lambda}})) \bigoplus_{t=s-r+1}^s p^{n-r-t}M_{p^t}(\mathbb{Q}(\zeta_{p^{m+r-s}})).
				\end{align*}
			\end{enumerate}
		\end{enumerate}
	\end{theorem}
	
	Several researchers have explored the complex irreducible representations and characters of metacyclic groups (see \cite{BGBasmaji, HK, HS, Munkholm}). In this paper, we explicitly compute all inequivalent complex irreducible representations for all ordinary metacyclic $p$-groups. Concurrently, we investigate the rational representations of these groups, which play a pivotal role in proving Theorem \ref{thm:wedderburnmetacyclic}. The structure of the paper is as follows. Let $G$ be an ordinary metacyclic $p$-group (as defined in \eqref{prest:metacyclic}). In Section \ref{sec:notation}, we introduce key notations that will be consistently used throughout the paper. Section \ref{sec:complexrep} focuses on the irreducible complex representations of $G$, while in Section \ref{sec:rationalrep}, we discuss results concerning the irreducible rational representations of $G$. In Theorem \ref{thm:rationalrep}, we present a combinatorial formula for enumerating the irreducible rational representations of $G$ with distinct degrees. In Subsection \ref{subsec:rationalmatrixrep}, we give method for computing the irreducible rational matrix representations of $G$. Further, Theorem \ref{thm:reqpairmeta} outlines techniques for determining all inequivalent irreducible matrix representations of $G$ over $\mathbb{Q}$. In Section \ref{sec:mainresult}, we present the proof of Theorem \ref{thm:wedderburnmetacyclic} and derive Corollary \ref{cor:isogroupalgebra} as an immediate consequence, which states that the rational group algebras of two ordinary metacyclic $p$-groups are isomorphic if and only if the groups themselves are isomorphic. Finally, in Subsection \ref{exam:metacyclic}, we conclude by providing examples that illustrate the Wedderburn decomposition of rational group algebras corresponding to various ordinary metacyclic $p$-groups, applying the different cases of Theorem \ref{thm:wedderburnmetacyclic}.
	
	
	\section{Notation}\label{sec:notation}
	In this section, we establish the notation, most of which adheres to standard conventions. Throughout this paper, $p$ denotes a prime number. For a finite group $G$, the following notation is used consistently.
	
	\begin{tabular}{cl}
		$G'$ & the derived (commutator) subgroup of $G$\\
		$|S|$ & the cardinality of a set $S$\\
		$\Irr(G)$ & the set of irreducible complex characters of $G$\\
		$\lin(G)$ & $\{\chi \in \Irr(G) : \chi(1) = 1\}$\\
		$\nl(G)$ & $\{\chi \in \Irr(G) : \chi(1) \neq 1\}$\\
		$\FIrr(G)$ & the set of faithful irreducible complex characters of $G$\\
		$\Irr^{(\lambda)}(G)$ & $\{\chi \in \Irr(G) : \chi(1) = \lambda\}$\\
		$\cd(G)$ & $\{ \chi(1) : \chi \in \Irr(G) \}$\\
		$\Irr_{\mathbb{Q}}(G)$ & the set of non-trivial irreducible rational characters of $G$\\
		$\Irr_{\mathbb{Q}}^{(\lambda)}(G)$ & $\{\chi \in \Irr_{\mathbb{Q}}(G) : \chi(1) = \lambda\}$\\
		$\mathbb{F}(\chi)$ & the smallest field containing the values $\{\chi(g) : g \in G\}$, where $\chi \in \Irr(G)$ and $\mathbb{F}$ is a base field\\
		$m_{\mathbb{F}}(\chi)$ & the Schur index of $\chi \in \Irr(G)$ over the field $\mathbb{F}$\\
		$\Omega(\chi)$ & $m_{\mathbb{Q}}(\chi)\sum_{\sigma \in \gal(\mathbb{Q}(\chi) / \mathbb{Q})} \chi^{\sigma}$, where $\chi \in \Irr(G)$\\
		$\cd_\mathbb{Q}(G)$ &  $\{ \Omega(\chi)(1) : \chi \in \Irr(G) \}$\\
		$\mathbb{F}G$ & the group algebra of $G$ over the field $\mathbb{F}$\\
		$M_q(D)$ & a full matrix ring of size $q$ over the skewfield $D$\\
		$Z(B)$ & the center of an algebraic structure $B$\\
		$\phi(q)$ & Euler's totient function
	\end{tabular}
	
	\section{Complex representations of ordinary metacyclic $p$-groups}\label{sec:complexrep}
	Basmaji \cite{BGBasmaji} and Munkholm \cite{Munkholm} obtained complex irreducible representations of metacyclic groups. In this section, we describe the complex irreducible representations of ordinary metacyclic $p$-groups (defined in \eqref{prest:metacyclic}) using the techniques outlined in \cite{BGBasmaji}, as it is highly utilized for upcoming sections. We begin with Lemma \ref{lemma:order}.

	\begin{lemma}\label{lemma:order}
		Let $p$ be a prime, and let $\alpha$ and $\beta$ be non-negative integers with $\alpha \neq 0$. Then the multiplicative order of $1+p^{\alpha}$ modulo $p^{\alpha+\beta}$ is $p^\beta$.
	\end{lemma}
	\begin{proof}
		For each non-zero integer $y$, define $w_p(y) = \max \{l : p^l \mid y\}$. Further, let $x=1+p^{\alpha}$. Then we need to show that $w_p(x^{p^\beta}-1)=\alpha + \beta$. We have 
		\begin{equation}\label{eq:1}
			\frac{x^{p^\beta}-1}{x-1}=\prod_{k=1}^{\beta}\Phi_{p^k}(x),
		\end{equation}
		where $\Phi_{p^k}(x)$ represents the $p^k$-th cyclotomic polynomial of $x$ over $\mathbb{Q}$. Furthermore, we have $\Phi_{p^k}(x)=1+x^{p^{k-1}}+x^{2p^{k-1}}+\cdots +x^{(p-1)p^{k-1}}$ for an odd prime $p$, and $\Phi_{2^k}(x)=1+x^{2^{k-1}}$. Observe that for $1 \leq k \leq \beta$, we have $w_p(\Phi_{p^k}(x))=1$, where $p$ is any prime. Therefore, from \eqref{eq:1}, we get
		\begin{align*}
			w_p\left(\frac{x^{p^\beta}-1}{x-1}\right)=\sum_{k=1}^{s}w_p(\Phi_{p^k}(x)) & \implies w_p(x^{p^\beta}-1)-w_p(x-1)=\beta\\
			& \implies w_p(x^{p^\beta}-1)-\alpha=\beta\\
			& \implies w_p(x^{p^\beta}-1)=\alpha+\beta.
		\end{align*}
		This completes the proof of Lemma \ref{lemma:order}.
	\end{proof}
	
	\begin{corollary}\label{cor:order}
		Let $p$ be a prime, and let $n$ and $s$ be non-negative integers such that $s <n$. Then the multiplicative order of $1+p^{n-s}$ modulo $p^n$ is $p^s$.
	\end{corollary}
	\begin{proof}
	By substituting $\alpha = n - s$ and $\beta = s$ in Lemma \ref{lemma:order}, the proof follows directly.
	\end{proof}
	Let $\zeta$ be a primitive $p^n$-th root of unity in $\mathbb{C}$ and let  $U :=\langle \zeta \rangle$. Suppose $\sigma : U \rightarrow U$ is defined as $\sigma(\zeta^i)=\zeta^{i(1+p^{n-s})}$ for $\zeta^i \in U$. This defines a natural group action of $\langle \sigma \rangle$ on $U$. Note that $\sigma^{p^s}(\zeta^i)=\zeta^{i(1+p^{n-s})^{p^s}}=\zeta^i$ as $1+p^{n-s}$ has multiplicative order $p^s$ modulo $p^n$ (see Corollary \ref{cor:order}). Thus, no orbit has size greater than $p^s$. Lemma \ref{lemma:sigmaorbitsize} describes the number of orbits of distinct size under the above action.

	
	\begin{lemma}\label{lemma:sigmaorbitsize}
		Under the above defined action of $\langle \sigma \rangle$ on $U$, we have the following.
		\begin{enumerate}
			\item Number of orbits of size $1$ is equal to $p^{n-s}$.
			\item Number of orbits of size $p^t$ is equal to $\phi(p^{n-s})$ for $1 \leq t \leq s$. 
		\end{enumerate}
	\end{lemma}
	\begin{proof}
		\begin{enumerate}
			\item Let $\zeta^x$ be a representative of an orbit of size $1$ for some $0 \leq x \leq p^n-1$. Hence, we have 
			\begin{align*}
				\sigma(\zeta^x)=\zeta^x & \implies \zeta^{(1+p^{n-s})x} = \zeta^x\\
				& \implies (1+p^{n-s})x \equiv x\pmod{p^n}\\
				& \implies p^{n-s}x \equiv 0\pmod{p^n}.
			\end{align*}
			Further, we have $\gcd\left(p^{n-s}, p^n\right)=p^{n-s}$. Therefore, the equation $p^{n-s}x \equiv 0\pmod{p^n}$ has only $p^{n-s}$ many incongruent solutions. Hence, the result follows.
			
			\item Let $\zeta^x$ be a representative of an orbit of size $p^t$ ($1 \leq t \leq s$) for some $0 \leq x \leq p^n-1$. Hence, we have 
			\begin{align*}
				\sigma^{p^t}(\zeta^x)=\zeta^x & \implies \zeta^{(1+p^{n-s})^{p^t}x} = \zeta^x\\
				& \implies (1+p^{n-s})^{p^t}x \equiv x\pmod{p^n}\\
				& \implies \left((1+p^{n-s})^{p^t}-1\right)x \equiv 0\pmod{p^n}.
			\end{align*}
			Further, we have $\gcd\left((1+p^{n-s})^{p^t}-1, p^n\right)=p^{n-s+t}$. Therefore, the equation $\left((1+p^{n-s})^{p^t}-1\right)x \equiv 0\pmod{p^n}$ has only $p^{n-s+t}$ many incongruent solutions. Furthermore, observe that 
			\begin{equation*}
				(1+p^{n-s})^{p^k}x \equiv x\pmod{p^n} \implies (1+p^{n-s})^{p^{k+1}}x \equiv x\pmod{p^n},
			\end{equation*}
			for all $0 \leq k \leq t-1$. Hence, the number of orbits of size $p^t$ is equal to $\frac{p^{n-s+t}-p^{n-s+t-1}}{p^t}=p^{n-s}-p^{n-s-1}=\phi(p^{n-s})$.
		\end{enumerate}
	\end{proof}

	Under the above set-up, in Remark \ref{remark:meta}, we give a brief description to construct all irreducible complex representations of ordinary metacyclic $p$-groups.

	\begin{remark}\label{remark:meta}\textnormal{ Let $G$ be an ordinary metacyclic $p$-group defined in \eqref{prest:metacyclic}. Then we have the following. 
		\begin{enumerate} 
				\item It is easy to check that
				$\{1\}, \{\zeta^{p^s}\}, \{\zeta^{2p^s}\}, \dots, \{\zeta^{(p^{n-s}-1)p^s}\}$
				are the only orbits of size one under the above defined action on $U$. Moreover, corresponding to the orbit $\mathcal{O}_{\zeta^{\lambda p^s}} = \{\zeta^{\lambda p^s}\}$ (for $0 \leq \lambda \leq p^{n-s} -1$) of size 1, the irreducible complex representations of $G$ can be defined as follows:
				\begin{equation}\label{eq:linearcomplexrep}
					T_{{\lambda p^s}, \omega}(a) = \zeta^{\lambda p^s} \quad \text{and} \quad T_{{\lambda p^s}, \omega}(b) = \omega,
				\end{equation}
				where $\zeta$ is any fixed primitive $p^n$-th root of unity and $\omega$ is a $p^m$-th root of unity. In this manner, we obtain $p^{n+m-s}=|G/G{}'|$ many degree 1 representations of $G$.
				\item For a fixed $t$ ($1\leq t \leq s$), we have
				$(1+p^{n-s})^{p^t} \equiv 1 \pmod{p^{n-s+t}}$ (see Lemma \ref{lemma:order}). This implies that $(1+p^{n-s})^{p^t}p^{s-t} \equiv p^{s-t}\pmod{p^n}$. Therefore, for a fixed $t$ ($1 \leq t\leq  s$),
				$$\mathcal{O}_{\zeta^{lp^{s-t}}}=\{\zeta^{lp^{s-t}}, \zeta^{(1+p^{n-s})lp^{s-t}}, \dots, \zeta^{(1+p^{n-s})^{p^t-1}lp^{s-t}} \},$$ 
				are the orbits of size $p^t$, where $1\leq l < p^{n-s+t}$ with $(l,p)=1$. Corresponding to an orbit $\mathcal{O}_{\zeta^{lp^{s-t}}}$, define a representation $T_{lp^{s-t}, \omega} : G \rightarrow GL_{p^t}(\mathbb{C})$ of degree $p^t$, given by
				\begin{equation}\label{eq:non-linearcomplexrep}
					\begin{aligned}
						T_{{lp^{s-t}}, \omega}(a) &= \left(
						\begin{array}{ccccc}
							\zeta^{lp^{s-t}} & 0 & 0 & \cdots & 0\\
							0 & \zeta^{(1+p^{n-s})lp^{s-t}} & 0 & \cdots & 0\\
							0 & 0 & \zeta^{(1+p^{n-s})^2lp^{s-t}} & \cdots & 0\\
							\vdots & \vdots & \vdots & \ddots \\
							0 & 0 & 0 & \cdots & \zeta^{(1+p^{n-s})^{p^t-1}lp^{s-t}}
						\end{array}\right) \quad \text{and}\\
						T_{{lp^{s-t}}, \omega}(b) &= \left(
						\begin{array}{ccccc}
							0 & 0 & \cdots  & 0 & \omega\\
							1 & 0 & \cdots & 0 & 0\\
							0 & 1 & \cdots & 0 & 0\\
							\vdots & \vdots & \ddots & & \vdots \\
							0 & 0 & \cdots & 1 & 0
						\end{array}\right),
					\end{aligned}
				\end{equation}
				where $\zeta$ is any fixed primitive $p^n$-th root of unity and $\omega^{p^{m-t}}=\zeta^{lp^{n+s-r-t}}$. This defines an irreducible complex representation of $G$ (see \cite{BGBasmaji}). 
		\end{enumerate}}
	\end{remark}
	
	Proposition \ref{prop:metacomplexrep} describes all inequivalent irreducible complex representations of an ordinary metacyclic $p$-group.
	\begin{proposition}\label{prop:metacomplexrep}
		Let $p$ be a prime. Consider an ordinary, finite, non-cyclic metacyclic $p$-group $G$, with a unique reduced presentation:
		\begin{equation*}
			G = \langle a, b ~ : ~ a^{p^n} = 1, \, b^{p^m} = a^{p^{n-r}}, \, bab^{-1} = a^{1+p^{n-s}} \rangle,
		\end{equation*}
		for certain integers $m, n, r$ and $s$ (as defined in \eqref{prest:metacyclic}). Then we have the following.
		\begin{enumerate}			
			\item $\cd(G)=\{p^t : 0 \leq t \leq s\}$.
			\item $|\lin(G)|= p^{n+m-s}$, and all linear complex representations of $G$ are given by \eqref{eq:linearcomplexrep}.
			\item For a given $t$ ($1 \leq t \leq s$), $|\Irr^{(p^t)}(G)|=\phi(p^{n-s})p^{m-t}$, and all inequivalent irreducible complex representations of degree $p^t$  are given by \eqref{eq:non-linearcomplexrep}.
			\item $|\Irr(G)|= p^{n+m-s} + p^{n+m-s-1} - p^{n+m-2s-1}$.
		\end{enumerate}
	\end{proposition}
	\begin{proof}
		\begin{enumerate}
			\item Clear.
			\item See Remark \ref{remark:meta} (1).
			\item By Lemma \ref{lemma:sigmaorbitsize}(2), for a fixed $t$ ($1 \leq t \leq s$), the number of orbits of size $p^t$ is $\phi(p^{n-s})$. Further, in Remark \ref{remark:meta}(2), the equation $\omega^{p^{m-t}}=\zeta^{lp^{n+s-r-t}}$ gives $p^{m-t}$ many distinct values of $\omega$. Thus, corresponding to an orbit of size $p^t$, we have $p^{m-t}$ many inequivalent irreducible complex representations of $G$ of degree $p^t$ (see \cite{BGBasmaji}). This proves the result. 
						
\item Here, \[|G| = p^{n+m-s} + \phi(p^{n-s})\sum_{t=1}^{s}p^{2t}p^{m-t}.\]
	Hence, the total number of inequivalent irreducible complex representations is $$p^{n+m-s} + \phi(p^{n-s})\sum_{t=1}^{s}p^{m-t}= p^{n+m-s} + p^{n+m-s-1} - p^{n+m-2s-1}.$$ \qedhere
		\end{enumerate}
	\end{proof}

	\section{Rational representations of ordinary metacyclic $p$-groups}\label{sec:rationalrep}
	In this section, we examine the irreducible rational representations of ordinary metacyclic $p$-groups $G$ (as defined in \eqref{prest:metacyclic}). To identify all the distinct irreducible rational representations of $G$, we divide them based on whether their kernels contain $G'$:
	\begin{enumerate}
		\item representations whose kernels contain $G'$; and
		\item representations whose kernels do not contain $G'$.
	\end{enumerate}
	We begin with the following Lemma.

	\begin{lemma}\label{lemma:abelian}
		For a given prime $p$, let $G = \langle a, b \mid  a^{p^n} = b^{p^m} = 1, ab = ba\rangle \cong C_{p^n} \times C_{p^m}$, where $n \geq m$. Then 
		the counting of non-trivial irreducible rational characters of $G$ of distinct degrees are the following. 	
	\begin{enumerate}	
	\item For each $\lambda$ with $1\leq \lambda \leq m$, $|\Irr_{\mathbb{Q}}^{(\phi(p^\lambda))}(G)|=p^{ \lambda -1 }(p + 1)$.
	\item For each $\lambda$ with $m+1 \leq \lambda \leq n$, $|\Irr_{\mathbb{Q}}^{(\phi(p^\lambda))}(G)|=p^{m}$.
		
		\end{enumerate}		
		\begin{proof} For any given prime $p$, the proof follows from \cite[Proposition 5]{Ram2}.		
\end{proof}	
	\end{lemma}
	

We now describe the irreducible rational representations of an ordinary, finite, non-cyclic metacyclic $p$-group $G$ (as defined in \eqref{prest:metacyclic}). Given that $G/G' = \langle aG', bG' \rangle \cong C_{p^{n-s}} \times C_{p^m}$, Lemma \ref{lemma:abelian} leads to two possible cases for the non-trivial irreducible rational representations of $G$ whose kernels contain $G'$. These cases are discussed in Remark \ref{remark:linear_metacyclic}.

	\begin{remark}\label{remark:linear_metacyclic}	
		\textnormal{{\bf Case 1 ($n-s \geq m$).} In this case, we have $|\Irr_{\mathbb{Q}}^{(\phi(p^\lambda))}(G)|=p^{ \lambda -1 }(p + 1)$ for each $\lambda$ with $1 \leq \lambda \leq m$, and $|\Irr_{\mathbb{Q}}^{(\phi(p^\lambda))}(G)|= p^m$ for each $\lambda$ with $m+1 \leq \lambda \leq n-s$.\\
			{\bf Case 2 ($n-s < m$).} In this case, we have $|\Irr_{\mathbb{Q}}^{(\phi(p^\lambda))}(G)|= p^{\lambda -1}(p+1)$ for each $\lambda$ with $1 \leq \lambda \leq n-s$, and $|\Irr_{\mathbb{Q}}^{(\phi(p^\lambda))}(G)|=p^{n-s}$  for each $\lambda$ with $n-s+1 \leq \lambda \leq m$.\\
			By Case 1 and Case 2, observe that $G$ has $\sum_{k=0}^{\min\{n-s, m\}}(n+m+1-s-2k)\phi(p^k)$ many inequivalent irreducible rational representations whose kernels contain $G'$.}
	\end{remark}	
	
	Next, we describe the irreducible rational representations of $G$ whose kernels do not contain $G'$. The character degrees of $G$ are given by $\cd(G) = \{p^t : 0 \leq t \leq s\}$. For a fixed $t$ ($1 \leq t \leq s$), there are $\phi(p^{n-s})p^{m-t}$ irreducible complex representations of $G$ with degree $p^t$ (see Proposition \ref{prop:metacomplexrep}), denoted by $T_{lp^{s-t}, \omega}$ (see \eqref{eq:non-linearcomplexrep}). Let $\chi$ denote the character corresponding to $T_{lp^{s-t}, \omega}$. We have 
	\begin{equation}\label{eq:charofb}
		T_{{lp^{s-t}}, \omega}(b) = \left(
		\begin{array}{ccccc}
			0 & 0 & \cdots  & 0 & \omega\\
			1 & 0 & \cdots & 0 & 0\\
			0 & 1 & \cdots & 0 & 0\\
			\vdots & \vdots & \ddots & & \vdots \\
			0 & 0 & \cdots & 1 & 0
		\end{array}\right),
	\end{equation}
	where $\omega^{p^{m-t}} = \zeta^{lp^{n+s-r-t}}$ and $\zeta$ is a fixed primitive $p^n$-th root of unity. For $1 \leq j < p^t$, by induction we can show that
	\begin{equation*}
		T_{{lp^{s-t}}, \omega}(b^j) = \left(\begin{array}{cc}
			O & \omega I_j\\
			I_{p^t-j} & O
		\end{array}\right),
	\end{equation*}
	where $I_j$ and $I_{p^t-j}$ are identity matrices of order $j$ and $p^t-j$, respectively, and $O$ is a zero matrix of appropriate size. Observe that any non-negative power of $T_{{lp^{s-t}}, \omega}(b)$ is either a diagonal matrix or has all diagonal entries equal to zero. Specifically, for any non-negative integer $\alpha$, $T_{{lp^{s-t}}, \omega}(b^{\alpha p^t}) = (T_{{lp^{s-t}}, \omega}(b))^{\alpha p^t} = \omega^\alpha I_{p^t}$, where $I_{p^t}$ is the identity matrix of order $p^t$. Thus, $\chi(b^{\alpha p^t}) = p^t\omega^\alpha$, and $\chi(b^\beta) = 0$ whenever $p^t \nmid \beta$. To determine $\chi(a^i)$ for any $i$, we first establish Lemma \ref{lemm:sumprimitveroots} and Lemma \ref{lemm:sumprimitveroots2}.
	
	\begin{lemma}\label{lemm:sumprimitveroots}
		Let $p$ be a prime, and let $\alpha, \beta$ be two non-negative integers such that $1 \leq \beta < \alpha$. Suppose $\zeta$ is a primitive $p^\alpha$-th root of unity. Then
		$$\sum_{i=0}^{p^\beta-1}\zeta^{(1+p^{\alpha-\beta})^i} = 0.$$
	\end{lemma}
	\begin{proof}
		Since $\alpha>\beta$, $\zeta^{p^{\alpha-\beta}}$ is a primitive $p^\beta$-th root of unity so that
		\[1+ \zeta^{p^{\alpha-\beta}} + \zeta^{2p^{\alpha-\beta}} + \cdots + \zeta^{(p^S-1)p^{\alpha-\beta}} = 0.\]
		Multiplying by $\zeta$, we have 
		\[\zeta + \zeta^{1+p^{\alpha-\beta}} + \zeta^{1+2p^{\alpha-\beta}} + \cdots + \zeta^{1+(p^\beta-1)p^{\alpha-\beta}} = 0.\]
		Further, from Lemma \ref{lemma:order}, the multiplicative order of $1+p^{\alpha-\beta}\pmod{p^\alpha}$ is $p^\beta$. Thus, $$|\{(1+p^{\alpha-\beta})^i\pmod{p^\alpha}~:~ 0\leq i\leq p^\beta-1\}|=p^\beta.$$ It follows that, for each $i$, $0 \leq i \leq p^\beta-1$, there is a unique $i{'}$,  $0 \leq i{'} \leq p^\beta-1$ such that $(1+p^{\alpha-\beta})^i \equiv (1+i{'}p^{\alpha-\beta})\pmod{p^\alpha}$. Hence, we have 
		$$\{(1+p^{\alpha-\beta})^i\pmod{p^\alpha}~:~ 0\leq i\leq p^\beta-1\}=\{(1+i{'}p^{\alpha-\beta})\pmod{p^\alpha}~:~ 0\leq i{'}\leq p^\beta-1\}.$$   	
		Therefore, we can rewrite the above expression as
		$$\zeta + \zeta^{1+p^{\alpha-\beta}} + \zeta^{(1+p^{\alpha-\beta})^2} +\cdots + \zeta^{(1+p^{\alpha-\beta})^{p^\beta-1}} = 0,$$
		and hence we get the required identity.
	\end{proof}
	
	\begin{lemma}\label{lemm:sumprimitveroots2}
			Let $p$ be a prime, and let $\alpha, \beta, \gamma$ and $\delta$ be non-negative integers such that $0\leq \delta < \beta < \alpha$ and $(\gamma, p^\alpha)= p^\delta$. Suppose $\zeta$ is a primitive $p^\alpha$-th root of unity. Then
		     $$\sum_{i=0}^{p^\beta-1}\zeta^{\gamma(1+p^{\alpha-\beta})^i} = 0.$$
	\end{lemma}
	\begin{proof}
		Since $(\gamma, p^\alpha)= p^\delta$, $\zeta^\gamma$ is a primitive $p^{\alpha-\delta}$-th root of unity. Therefore, by substituting $\alpha-\delta$ and $\beta-\delta$ for $\alpha$ and $\beta$, respectively in Lemma \ref{lemm:sumprimitveroots}, we get
		$$\sum_{i=0}^{p^{\beta-\delta}-1}\zeta^{\gamma(1+p^{\alpha-\beta})^i} = 0.$$
		Further, from Lemma \ref{lemma:order}, the multiplicative order of $1+p^{\alpha-\beta}\pmod{p^{\alpha-\delta}}$ is $p^{\beta-\delta}$. Thus, for $0 \leq i \leq p^\beta-1$ and $0 \leq j \leq p^{\beta-\delta}-1$ such that $i \equiv j \pmod{p^{\beta-\delta}}$, we have
		$(1+p^{\alpha-\beta})^i \equiv (1+p^{\alpha-\beta})^j\pmod{p^{\alpha-\delta}}$. It follows that $\zeta^{\gamma(1+p^{\alpha-\beta})^i}=\zeta^{\gamma(1+p^{\alpha-\beta})^j}$ and hence we get 
		$$\sum_{i=0}^{p^\beta-1}\zeta^{\gamma(1+p^{\alpha-\gamma})^i} = p^\delta\sum_{i=0}^{p^{\beta-\delta}-1}\zeta^{\gamma(1+p^{\alpha-\beta})^i} = 0.$$
	\end{proof}
	
Now, we are ready to determine the value of $\chi(a^i)$ for any $i$, where $\chi$ is the character afforded by the representation $T_{lp^{s-t}, \omega}$ defined in \eqref{eq:non-linearcomplexrep}. Here, we have
	\begin{equation}\label{eq:charofa}
		T_{{lp^{s-t}}, \omega}(a^k) = \left(
		\begin{array}{ccccc}
			\zeta^{klp^{s-t}} & 0 & 0 & \cdots & 0\\
			0 & \zeta^{k(1+p^{n-s})lp^{s-t}} & 0 & \cdots & 0\\
			0 & 0 & \zeta^{k(1+p^{n-s})^2lp^{s-t}} & \cdots & 0\\
			\vdots & \vdots & \vdots & \ddots & \vdots\\
			0 & 0 & 0 & \cdots & \zeta^{k(1+p^{n-s})^{p^t-1}lp^{s-t}}
		\end{array}\right),
	\end{equation}
	where $\zeta$ is a primitive $p^n$-th root of unity. Observe that $T_{{lp^{s-t}}, \omega}(a^{p^t}) = \zeta^{lp^s}I_{p^t}$, where $I_{p^t}$ is the identity matrix of order $p^t$. Hence, $\chi(a^{p^t}) = p^t\zeta^{lp^s}$. Furthermore, $\zeta^{lp^{s-t}}$ is a primitive $p^{n-s+t}$-th root of unity. By substituting $\alpha = n-s+t$ and $\beta = t$ in Lemma \ref{lemm:sumprimitveroots2}, we get
	$$\chi(a^k) = \sum_{i=0}^{p^t-1}\zeta^{k(1+p^{n-s})^ilp^{s-t}} = 0$$
	for all $0 \leq k \leq p^n-1$ such that $p^t \nmid k$. Thus, we have
	$$\chi(a^i) = \begin{cases}
		p^t\zeta^{ilp^{s-t}} &\quad \text{if } p^t\mid i,\\
		0            &\quad \text{otherwise.}
	\end{cases}$$
	Therefore, the character values $\chi(a^ib^j)$ for all elements of $G$ are given by
	\begin{equation}\label{eq:chatractervalue}
		\chi(a^ib^j) = \begin{cases}
			p^t\omega^{j_1}\zeta^{ilp^{s-t}} &\quad \text{if } p^t \mid i \text{ and } j = j_1p^t,\\
			0            &\quad \text{otherwise,}
		\end{cases}
	\end{equation}
		where $\zeta$ is any fixed primitive $p^n$-th root of unity and $\omega^{p^{m-t}}=\zeta^{lp^{n+s-r-t}}$.
	
	Now, we define an equivalence relation on $\Irr(G)$ based on Galois conjugacy over $\mathbb{Q}$. For $\chi, \psi \in \Irr(G)$, we say that $\chi$ and $\psi$ are Galois conjugates over $\mathbb{Q}$ if $\mathbb{Q}(\chi) = \mathbb{Q}(\psi)$ and there exists $\sigma \in \gal(\mathbb{Q}(\chi) / \mathbb{Q})$ such that $\chi^\sigma = \psi$.
	
	\begin{lemma}\cite[Lemma 9.17]{I}\label{SC}
		Let $E(\chi)$ denote the Galois conjugacy class over $\mathbb{Q}$ of a complex irreducible character $\chi$. Then 
		\[|E(\chi)| = [\mathbb{Q}(\chi) : \mathbb{Q}]. \]
	\end{lemma}
	
	For $1 \leq t \leq s-r$, $\omega$ in \eqref{eq:non-linearcomplexrep} is a primitive $p^\lambda$-th root of unity, where $0 \leq \lambda \leq m-t$. Hence,
	$$\gal(\mathbb{Q}(\chi) / \mathbb{Q}) = \gal(\mathbb{Q}(\eta) / \mathbb{Q}),$$
	where $\eta$ is a primitive $p^{n-s}$-th root of unity if $n-s \geq \lambda$, and a primitive $p^\lambda$-th root of unity if $n-s < \lambda$. Furthermore, when $r \neq 0$ and $s-r < t \leq s$, $\omega$ is a primitive $p^{m+r-s}$-th root of unity in \eqref{eq:non-linearcomplexrep}. Based on the above discussion and Lemma \ref{SC}, we summarize the counting of Galois conjugacy classes of those irreducible complex characters of $G$ whose kernels do not contain $G'$.
	
	\begin{remark}\label{remark:galoisclasses_metacyclic} 
		\textnormal{{\bf Case 1.} For a fixed $t$ ($1\leq t \leq s-r$), we have one of the following two sub-cases for Galois conjugacy classes of irreducible complex characters of degree $p^t$ of $G$.\\
			{\bf Sub-Case 1 ($n-s \geq m-t$).} There are $\sum_{\lambda = 0}^{m-t}\phi(p^\lambda) = p^{m-t}$ many distinct Galois conjugacy classes of size $\phi(p^{n-s})$.\\
			{\bf Sub-Case 2 ($n-s < m-t$).} There are $\sum_{\lambda = 0}^{n-s}\phi(p^\lambda) = p^{n-s}$ many distinct Galois conjugacy classes of size $\phi(p^{n-s})$ and $\phi(p^{n-s})$ many distinct Galois conjugacy classes of size $\phi(p^\lambda)$ for $n-s < \lambda \leq m-t$.\\
			{\bf Case 2.} As either $r=0$ or $n-m+1 \leq r$ (see \eqref{prest:metacyclic}), for a fixed $t$ where $s-r+1 \leq t \leq s$, there are $\frac{\phi(p^{n-s})p^{m-t}}{\phi(p^{m+r-s})}=p^{n-r-t}$  many distinct Galois conjugacy classes of size $\phi(p^{m+r-s})$ of irreducible complex characters of degree $p^t$ of $G$.}
	\end{remark}
	
	In Lemma \ref{lemma:schurindex}, we calculate the Schur indices of irreducible complex characters over $\mathbb{Q}$ for ordinary metacyclic $p$-groups.
	
	\begin{lemma}\cite[Corollary 10.14]{I}\label{lemma:schur}
		Let $G$ be a $p$-group and $\chi \in \Irr(G)$. For any field $\mathbb{F} \subseteq \mathbb{C}$, the Schur index $m_\mathbb{F}(\chi)$ equals 1 unless $p = 2$ and $\sqrt{-1} \notin \mathbb{F}$, in which case $m_\mathbb{F}(\chi) \leq 2$.
	\end{lemma}
	
	\begin{lemma}\label{lemma:schurindex}
		Let $G$ be an ordinary metacyclic $p$-group and $\chi \in \Irr(G)$. Then $m_\mathbb{Q}(\chi) = 1$.
	\end{lemma}
	
	\begin{proof}
		If $p$ is odd, the result follows from Lemma \ref{lemma:schur}. Next, consider an ordinary metacyclic $2$-group $G$ with a unique reduced presentation (as given in \eqref{prest:metacyclic}), and let $\chi \in \Irr(G)$. Since $n - s \geq 2$ and $m + r - s \geq 2$, we have $\sqrt{-1} \in \mathbb{Q}(\chi)$ (see \eqref{eq:chatractervalue}). By Lemma \ref{lemma:schur}, this implies that $m_{\mathbb{Q}(\chi)}(\chi) = 1$. Moreover, $m_\mathbb{Q}(\chi) = m_{\mathbb{Q}(\chi)}(\chi)$ (see \cite[Corollary 10.2]{I}), so $m_\mathbb{Q}(\chi) = 1$. This completes the proof of Lemma \ref{lemma:schurindex}.
	\end{proof}
	
	Additionally, for an ordinary metacyclic $p$-group $G$, Schur's theory guarantees the existence of a unique irreducible $\mathbb{Q}$-representation $\rho$ of $G$ such that
	$$\rho = \bigoplus_{\sigma \in \gal(\mathbb{Q}(\chi) / \mathbb{Q})} (T_{lp^{s-t}, \omega})^\sigma$$
	with $\deg \rho = \chi(1) |\gal(\mathbb{Q}(\chi) / \mathbb{Q})|$. Thus, the number of irreducible rational representations of $G$ equals the number of Galois conjugacy classes in $\Irr(G)$. Based on Remark \ref{remark:linear_metacyclic} and Remark \ref{remark:galoisclasses_metacyclic}, Theorem \ref{thm:rationalrep} provides the count of irreducible rational representations of an ordinary metacyclic $p$-group.
	
	\begin{theorem}\label{thm:rationalrep}
		Let $p$ be a prime. Consider an ordinary, finite, non-cyclic metacyclic $p$-group $G$ with a unique reduced presentation:
		\begin{equation*}
			G = \langle a, b ~ : ~ a^{p^n} = 1, \, b^{p^m} = a^{p^{n-r}}, \, bab^{-1} = a^{1+p^{n-s}} \rangle,
		\end{equation*}
		for certain integers $m, n, r$ and $s$ (as defined in  \eqref{prest:metacyclic}). Then we have the following two cases, which determine the counting of non-trivial irreducible rational representations of $G$ of distinct degrees.
		\begin{enumerate}
			\item {\bf Case ($n-s \geq m$).} In this case, we have
			\begin{itemize}
			\item 	$\cd_\mathbb{Q}(G)=\{\phi(p^\theta) : 0 \leq \theta \leq n\}$,	
			\item $|\Irr_{\mathbb{Q}}^{(\phi(p^\lambda))}(G)|= p^{\lambda - 1}(p+1)$ for each $\lambda$ with $1 \leq \lambda \leq m$,
			\item $|\Irr_{\mathbb{Q}}^{(\phi(p^\lambda))}(G)|=p^m$ for each $\lambda$ with $m+1\leq \lambda \leq n-s$,
			\item $|\Irr_{\mathbb{Q}}^{(\phi(p^\lambda))}(G)|=p^{m-t}$ for each $\lambda = n-s+t$ with $1 \leq t \leq s$.
			\end{itemize}		

			\item {\bf Case ($n-s < m$).} Suppose $m = (n-s)+k$. In this case, we have two sub-cases.
			\begin{enumerate}
				\item {\bf Sub-case ($k \leq s-r$).} In this sub-case, we have
				\begin{itemize}
				\item 	$\cd_\mathbb{Q}(G)=\{\phi(p^\theta) : 0 \leq \theta \leq n\}$,	
				\item $|\Irr_{\mathbb{Q}}^{(\phi(p^\lambda))}(G)|=p^{\lambda - 1}(p+1)$ for each $\lambda$ with $1 \leq \lambda \leq n-s$,
				\item $|\Irr_{\mathbb{Q}}^{(\phi(p^\lambda))}(G)|=2p^{n-s} + (t-1)\phi(p^{n-s})$ for each $\lambda = n-s+t$ with $1 \leq t \leq k$,
				\item  $|\Irr_{\mathbb{Q}}^{(\phi(p^\lambda))}(G)|=p^{m-t}$ for each $\lambda = n-s+t$ with $k+1 \leq t \leq s$.
				\end{itemize}

				\item {\bf Sub-case ($k > s-r$).} In this sub-case, we have 
				\begin{itemize}
				\item 	$\cd_\mathbb{Q}(G)=\{\phi(p^\theta) : 0 \leq \theta \leq m+r\}$,	
				\item $|\Irr_{\mathbb{Q}}^{(\phi(p^\lambda))}(G)|=p^{\lambda - 1}(p+1)$ for each $\lambda$ with $1 \leq \lambda \leq n-s$,
				\item $|\Irr_{\mathbb{Q}}^{(\phi(p^\lambda))}(G)|=2p^{n-s} + (t-1)\phi(p^{n-s})$ for each $\lambda = n-s+t$ with $1 \leq t \leq s-r$,
				\item $|\Irr_{\mathbb{Q}}^{(\phi(p^\lambda))}(G)|=p^{n-s}+(s-r)\phi(p^{n-s})$ for each $\lambda$ with $n-r+1 \leq \lambda \leq m$, 
				\item  $|\Irr_{\mathbb{Q}}^{(\phi(p^\lambda))}(G)|=p^{n-r-t}$ for each $\lambda = m+r-s+t$ with $s-r+1 \leq t \leq s$.
				\end{itemize}

			\end{enumerate}
		\end{enumerate}
	\end{theorem}
	\begin{proof}Let $G$ be an ordinary, finite, non-cyclic metacyclic $p$-group with a unique reduced presentation given in \eqref{prest:metacyclic}.
		\begin{enumerate}
			\item When $n - s \geq m$, as described in Case 1 of Remark \ref{remark:linear_metacyclic}, $G$ possesses $p^{\lambda - 1}(p + 1)$ inequivalent non-trivial irreducible rational representations of degree $\phi(p^\lambda)$ for $1 \leq \lambda \leq m$. Additionally, for each $\lambda$ with $m+1\leq  \lambda \leq n - s$, there are $p^m$ inequivalent irreducible rational representations of degree $\phi(p^\lambda)$. In this scenario, $G$ is always split, meaning $r=0$ (refer to \eqref{prest:metacyclic} for further details). Since $n - s > m - t$ for $1 \leq t \leq s$, from Case 1 (Sub-case 1) of Remark \ref{remark:galoisclasses_metacyclic}, there exist $p^{m - t}$ distinct Galois conjugacy classes of irreducible complex characters of degree $p^t$, each with size $\phi(p^{n - s})$. Consequently, $G$ has $p^{m - t}$ inequivalent irreducible rational representations of degree $p^t \phi(p^{n - s}) = \phi(p^{n - s + t})$ for each $1 \leq t \leq s$. This completes the proof of Theorem \ref{thm:rationalrep}(1).
			
			\item When $n - s < m$, as detailed in Case 2 of Remark \ref{remark:linear_metacyclic}, $G$ has $p^{\lambda - 1}(p + 1)$ inequivalent non-trivial irreducible rational representations of degree $\phi(p^\lambda)$ for $1 \leq \lambda \leq n - s$, and $p^{n - s}$ inequivalent irreducible rational representations of degree $\phi(p^\lambda)$ for each $\lambda$ with $n - s < \lambda \leq m$. Suppose $m = n - s + k$. We then have two sub-cases to consider.
			\begin{enumerate}
				\item Suppose $k \leq s - r$. This implies $m - n + s \leq s - r$, leading to $r \leq n - m$. Thus, in this sub-case, $G$ is always split, meaning $r = 0$ (refer to \eqref{prest:metacyclic} for more details), and $m \leq n$.
				
				Now, for $1 \leq t \leq k - 1$, we have $n - s < m - t$. Therefore, by Case 1 (Sub-case 2) of Remark \ref{remark:galoisclasses_metacyclic}, there are $p^{n - s}$ distinct Galois conjugacy classes of size $\phi(p^{n - s})$ and $\phi(p^{n - s})$ distinct Galois conjugacy classes of size $\phi(p^\lambda)$ for $n - s+1 \leq \lambda \leq m - t$, corresponding to irreducible complex characters of degree $p^t$ for a fixed $t$ (where $1 \leq t \leq k - 1$). Therefore, $G$ has $p^{n - s}$ distinct irreducible rational representations of degree $p^t \phi(p^{n - s}) = \phi(p^{n - s + t})$ and $\phi(p^{n - s})$ distinct irreducible rational representations of degree $p^t \phi(p^\lambda) = \phi(p^{\lambda + t})$ for $n - s+1 \leq \lambda \leq m - t$, for a fixed $t$ (where $1 \leq t \leq k - 1$).
				
				Further, for $k \leq t \leq s$, we have $n - s \geq m - t$. Therefore, by Case 1 (Sub-case 1) of Remark \ref{remark:galoisclasses_metacyclic}, there are $p^{m - t}$ distinct Galois conjugacy classes of size $\phi(p^{n - s})$ corresponding to irreducible complex characters of degree $p^t$ for a fixed $t$ (where $k \leq t \leq s$). Thus, $G$ has $p^{m - t}$ distinct irreducible rational representations of degree $p^t \phi(p^{n - s}) = \phi(p^{n - s + t})$ for a fixed $t$ (where $k \leq t \leq s$).
				
				To make the counting clearer in this sub-case, we list the irreducible $\mathbb{Q}$-representations of $G$ whose kernels not contain $G'$, along with their distinct degrees, corresponding to the Galois conjugacy classes of irreducible complex characters of distinct degrees, separately in Table \ref{t:1}. 
		\begin{longtable}[c]{|c|c|c|}
			\caption{Irreducible $\mathbb{Q}$-representations of $G$ whose kernel do not contain $G'$ under sub-case $(k\leq s-r)$ \label{t:1}}\\
			\hline
			\pbox{20cm}{Value of $t$\\ $(1\leq t\leq s)$} & \pbox{20cm}{ Degree of irreducible $\mathbb{Q}$-representations of $G$ \\ corresponding to the Galois conjugacy classes \\of irreducible complex characters of $G$ with\\ degree $p^t$} & \pbox{20cm}{Number of irreducible \\ $\mathbb{Q}$-representations of the\\ degree specified in the\\ second column} \\[9mm]
			\hline
			\endfirsthead
			\hline
			\pbox{20cm}{Value of $t$\\ $(1\leq t\leq s)$} & \pbox{20cm}{ Degree of irreducible $\mathbb{Q}$-representations of $G$ \\ corresponding to the Galois conjugacy classes \\of irreducible complex characters of $G$ with\\ degree $p^t$} & \pbox{20cm}{Number of irreducible \\ $\mathbb{Q}$-representations of the\\ degree specified in the\\ second column} \\[9mm]
			\hline
			\endhead
			\hline
			\endfoot
			\hline
			\endlastfoot
			\hline \multirow{5}{*}{$t=1$} & $p\phi(p^{n-s})=\phi(p^{n-s+1})$ & $p^{n-s}$  \\
			& $p\phi(p^{n-s+1})=\phi(p^{n-s+2})$ & $\phi(p^{n-s})$  \\
			& $p\phi(p^{n-s+2})=\phi(p^{n-s+3})$ & $\phi(p^{n-s})$ \\
			& $\vdots$ & $\vdots$  \\
			& $p\phi(p^{m-1})=\phi(p^m)$ & $\phi(p^{n-s})$   \\ \hline
			
			\multirow{5}{*}{$t=2$} & $p^2\phi(p^{n-s})=\phi(p^{n-s+2})$ & $p^{n-s}$  \\
			& $p^2\phi(p^{n-s+1})=\phi(p^{n-s+3})$ & $\phi(p^{n-s})$  \\
			& $p^2\phi(p^{n-s+2})=\phi(p^{n-s+4})$ & $\phi(p^{n-s})$ \\
			& $\vdots$ & $\vdots$  \\
			& $p^2\phi(p^{m-2})=\phi(p^m)$ & $\phi(p^{n-s})$   \\ \hline
			
			$\vdots$ & $\vdots$ & $\vdots$ \\ \hline
			
			\multirow{3}{*}{$t=k-2$} & $p^{k-2}\phi(p^{n-s})=\phi(p^{n-s+k-2})$ & $p^{n-s}$  \\
			& $p^{k-2}\phi(p^{n-s+1})=\phi(p^{n-s+k-1})=\phi(p^m)$ & $\phi(p^{n-s})$\\
			& $p^{k-2}\phi(p^{n-s+2})=\phi(p^{n-s+k})=\phi(p^m)$ & $\phi(p^{n-s})$  \\ \hline
			
			\multirow{2}{*}{$t=k-1$} & $p^{k-1}\phi(p^{n-s})=\phi(p^{n-s+k-1})$ & $p^{n-s}$  \\
			& $p^{k-1}\phi(p^{n-s+1})=\phi(p^{n-s+k})=\phi(p^m)$ & $\phi(p^{n-s})$  \\ \hline
			
			$t=k$ & $p^{k}\phi(p^{n-s})=\phi(p^{n-s+k})=\phi(p^m)$ & $p^{m-k}$  \\ \hline
			
			$t=k+1$ & $p^{k+1}\phi(p^{n-s})=\phi(p^{n-s+k+1})=\phi(p^{m+1})$ & $p^{m-k-1}$  \\ \hline
			
			$\vdots$ & $\vdots$ & $\vdots$ \\ \hline
			
			$t=s$ & $p^{s}\phi(p^{n-s})=\phi(p^{n})$ & $p^{m-s}$  \\ \hline
		\end{longtable}
		
		The result follows by combining all the irreducible rational representations of $G$ with distinct degrees, as discussed in this sub-case, with the aid of Table \ref{t:1}.
				
				\item Suppose $k > s - r$. In this case, $m > n - r$, and $n - s < m - t$ for $1 \leq t \leq s - r$. According to Sub-case 2 of Case 1 of Remark \ref{remark:galoisclasses_metacyclic}, there are $p^{n - s}$ distinct Galois conjugacy classes of size $\phi(p^{n - s})$ and $\phi(p^{n - s})$ distinct Galois conjugacy classes of size $\phi(p^\lambda)$ for $n - s < \lambda \leq m - t$ of irreducible complex characters of degree $p^t$ for a fixed $t$ (where $1 \leq t \leq s - r$). Therefore, $G$ has $p^{n - s}$ distinct irreducible rational representations of degree $p^t \phi(p^{n - s}) = \phi(p^{n - s + t})$ and $\phi(p^{n - s})$ distinct irreducible rational representations of degree $p^t \phi(p^\lambda) = \phi(p^{\lambda + t})$ for $n - s < \lambda \leq m - t$ for a fixed $t$ (where $1 \leq t \leq s - r$). Additionally, from Case 2 of Remark \ref{remark:galoisclasses_metacyclic}, for a fixed $t$ where $s - r + 1 \leq t \leq s$, there are $p^{n - r - t}$ distinct Galois conjugacy classes of irreducible complex characters of degree $p^t$, each with size $\phi(p^{m + r - s})$. Hence, $G$ has $p^{n - r - t}$ inequivalent irreducible rational representations of degree $p^t \phi(p^{m + r - s}) = \phi(p^{m + r - s + t})$ for each $s - r + 1 \leq t \leq s$. 
				
				Similar to Table \ref{t:1} in the previous sub-case, Table \ref{t:2} lists the irreducible $\mathbb{Q}$-representations of $G$ in this sub-case whose kernels do not contain $G'$, along with their distinct degrees, corresponding to the Galois conjugacy classes of irreducible complex characters of distinct degrees.
				
				\begin{longtable}[c]{|c|c|c|}
					\caption{Irreducible $\mathbb{Q}$-representations of $G$ whose kernel do not contain $G'$ under sub-case $(k > s-r)$ \label{t:2}}\\
					\hline
					\pbox{20cm}{Value of $t$\\ $(1\leq t\leq s)$} & \pbox{20cm}{ Degree of irreducible $\mathbb{Q}$-representations of $G$ \\ corresponding to the Galois conjugacy classes \\of irreducible complex characters of $G$ with\\ degree $p^t$} & \pbox{20cm}{Number of irreducible \\ $\mathbb{Q}$-representations of the\\ degree specified in the\\ second column} \\[9mm]
					\hline
					\endfirsthead
					\hline
					\pbox{20cm}{Value of $t$\\ $(1\leq t\leq s)$} & \pbox{20cm}{ Degree of irreducible $\mathbb{Q}$-representations of $G$ \\ corresponding to the Galois conjugacy classes \\of irreducible complex characters of $G$ with\\ degree $p^t$} & \pbox{20cm}{Number of irreducible \\ $\mathbb{Q}$-representations of the\\ degree specified in the\\ second column} \\[9mm]
					\hline
					\endhead
					\hline
					\endfoot
					\hline
					\endlastfoot
					\hline \multirow{5}{*}{$t=1$} & $p\phi(p^{n-s})=\phi(p^{n-s+1})$ & $p^{n-s}$  \\
					& $p\phi(p^{n-s+1})=\phi(p^{n-s+2})$ & $\phi(p^{n-s})$  \\
					& $p\phi(p^{n-s+2})=\phi(p^{n-s+3})$ & $\phi(p^{n-s})$ \\
					& $\vdots$ & $\vdots$  \\
					& $p\phi(p^{m-1})=\phi(p^m)$ & $\phi(p^{n-s})$   \\ \hline
					
					\multirow{5}{*}{$t=2$} & $p^2\phi(p^{n-s})=\phi(p^{n-s+2})$ & $p^{n-s}$  \\
					& $p^2\phi(p^{n-s+1})=\phi(p^{n-s+3})$ & $\phi(p^{n-s})$  \\
					& $p^2\phi(p^{n-s+2})=\phi(p^{n-s+4})$ & $\phi(p^{n-s})$ \\
					& $\vdots$ & $\vdots$  \\
					& $p^2\phi(p^{m-2})=\phi(p^m)$ & $\phi(p^{n-s})$   \\ \hline
					
					$\vdots$ & $\vdots$ & $\vdots$ \\ \hline
					
					\multirow{5}{*}{$t=s-r$} & $p^{s-r}\phi(p^{n-s})=\phi(p^{n-r})$ & $p^{n-s}$  \\
					& $p^{s-r}\phi(p^{n-s+1})=\phi(p^{n-r+1})$ & $\phi(p^{n-s})$  \\
					& $p^{s-r}\phi(p^{n-s+2})=\phi(p^{n-r+2})$ & $\phi(p^{n-s})$ \\
					& $\vdots$ & $\vdots$  \\
					& $p^{s-r}\phi(p^{m-s+r})=\phi(p^m)$ & $\phi(p^{n-s})$   \\ \hline
					
					$t=s-r+1$ & $p^{s-r+1}\phi(p^{m+r-s})=\phi(p^{m+1})$ & $p^{n-s-1}$  \\ \hline
					
					$t=s-r+2$ & $p^{s-r+2}\phi(p^{m+r-s})=\phi(p^{m+2})$ & $p^{n-s-2}$  \\ \hline
					
					$\vdots$ & $\vdots$ & $\vdots$ \\ \hline
					
					$t=s$ & $p^{s}\phi(p^{m+r-s})=\phi(p^{m+r})$ & $p^{n-s-r}$  \\ \hline
					
				\end{longtable}
				
			By combining all the irreducible rational representations of $G$ discussed in this sub-case, with the help of Table \ref{t:2}, we obtain the result.
			\end{enumerate}
			This completes the proof of Theorem \ref{thm:rationalrep}(2). \qedhere
		\end{enumerate}
	\end{proof}	
	
	\subsection{Rational matrix representations of ordinary metacyclic $p$-groups} \label{subsec:rationalmatrixrep}
	 Here, we present a method for computing irreducible matrix representations of an ordinary metacyclic $p$-group over $\mathbb{Q}$. Iida and Yamada \cite{YI} investigated rational matrix representations of metacyclic $p$-groups for an odd prime $p$. Our results extend their work by determining irreducible matrix representations for all ordinary metacyclic $p$-groups, where $p$ is any prime. Before that, we prove some properties of an ordinary metacyclic $p$-group.

	\begin{lemma}\label{lemma:quotientsubgroupmeta}
		The subgroups and quotients of a finite metacyclic group are also metacyclic.
	\end{lemma}
	\begin{proof}
		Let $G$ be a finite metacyclic group with $N \trianglelefteq G$, where both $N$ and $G/N$ are cyclic. Consider any subgroup $H < G$. By the Second Isomorphism Theorem, we have
		$$(HN)/N \cong H/(H \cap N).$$
		Since $H \cap N$ is a subgroup of $N$, $H \cap N$ is also cyclic. Additionally, $H/(H \cap N)$ is cyclic because $(HN)/N$ is a subgroup of the cyclic group $G/N$. Therefore, $H$ is metacyclic.\\
		Next, consider the quotient group $G/H$, where $H \trianglelefteq G$. We claim that $(HN)/H \trianglelefteq G/H$, with both $(HN)/H$ and $(G/H)/((HN)/H)$ cyclic. The group $(HN)/H$ is cyclic, as it is the image of the cyclic group $N$ under the canonical projection $G \to G/H$. Furthermore, since both $H$ and $N$ are normal in $G$, it follows that $HN \trianglelefteq G$ and $NH \trianglelefteq G$. By the Third Isomorphism Theorem, we obtain
		$${(HN)}/H \trianglelefteq G/H\,\,\, \text{and}\,\,\, {(G/H)}/{({(HN)}/H)} \cong G/{(HN)} \cong {(G/N)}/{((NH)/N)}.$$
		Since ${(G/N)}/{((NH)/N)}$ is cyclic, $(G/H)/((HN)/H)$ is also cyclic. Thus, $G/H$ is metacyclic. This completes the proof of Lemma \ref{lemma:quotientsubgroupmeta}.
	\end{proof}
Recall that a finite group $G$ is said to be faithful if it has a faithful irreducible complex character.	
Corollary \ref{cor:quotientmeta} asserts that the study of rational matrix representations of an ordinary metacyclic $p$-group can be reduced to the study of rational matrix representations of all faithful ordinary metacyclic $p$-groups.

	\begin{corollary}\label{cor:quotientmeta}
		For any prime $p$, the quotients of an ordinary, finite, metacyclic $p$-group are also ordinary metacyclic.
	\end{corollary}
	
	\begin{proof}
		If $p$ is odd, then the result directly follows from Lemma \ref{lemma:quotientsubgroupmeta} and \eqref{prest:metacyclic}. Now, let $G$ be an ordinary, finite, metacyclic $2$-group (as defined in \eqref{prest:metacyclic}). On the contrary, suppose,  $G/H$ is not ordinary metacyclic for any given normal subgroup $H$ of $G$, i.e., it is exceptional metacyclic (see \eqref{prest:exmetacyclic}). Consider the unique reduced presentation of $G/H$ as follows:
		\begin{equation*}
			G/H = \langle a, b ~ : ~ a^{2^n} = 1, b^{2^m} = a^{2^{n-r}}, bab^{-1} = a^{-1+2^{n-s}} \rangle
		\end{equation*}
		for certain integers $m$, $n$, $r$, and $s$ (as given in \eqref{prest:exmetacyclic}). Define a homomorphism $\rho : G/H \rightarrow GL_2(\mathbb{C})$ as follows:
		\begin{equation*}
			\rho(a) = \left(
			\begin{array}{cc}
				\zeta_4 & 0\\
				0 & \zeta_4^{-1}\\
			\end{array}\right) \, \, \, \text{and} \, \, \, \rho(b) = \left(
			\begin{array}{cc}
				0 & 1\\
				1 & 0\\
			\end{array}\right),
		\end{equation*}
		where $\zeta_4$ is a primitive $4$-th root of unity. Observe that $\rho$ is an irreducible complex representation of $G/H$. Let $\bar{\chi}$ denote the character corresponding to $\rho$, and let $\chi$ be the lift of $\bar{\chi}$ to $G$. Then we have $\mathbb{Q}(\chi)=\mathbb{Q}$. This implies that $G$ has an irreducible complex representation of degree $2$ affording the character $\chi$ such that $\mathbb{Q}(\chi)=\mathbb{Q}$, which is a contradiction, because $\sqrt{-1} \in \mathbb{Q}(\chi)$ (see \eqref{eq:chatractervalue}). Hence, $G/H$ cannot be an exceptional metacyclic $2$-group. Therefore, $G/H$ is an ordinary metacyclic $2$-group. This completes the proof of Corollary \ref{cor:quotientmeta}.
	\end{proof}
	
	 Lemma \ref{lemma:faithfulmeta} classifies all faithful ordinary metacyclic $p$-groups.
	\begin{lemma}\label{lemma:faithfulmeta}
		Let $G$ be a faithful, ordinary, finite, non-cyclic metacyclic $p$-group. Then $G$ has, up to isomorphism, exactly one uniquely reduced presentation, which is one of the following two types:
		\begin{enumerate}
			\item $G_1 = \langle a, b ~ : ~ a^{p^n} = 1, \, b^{p^m} = a^{p^{s}}, \, bab^{-1} = a^{1+p^{n-s}} \rangle$ for certain integers $n$, $m$ and $s$, where $\max\{1, n-m+1\} \leq n-s < s$ for $p \geq 3$, and $\max\{2, n-m+1\} \leq n-s < s$ for $p=2$.
			\item $G_2 = \langle a, b ~ : ~ a^{p^n} = b^{p^s} = 1, \, bab^{-1} = a^{1+p^{n-s}} \rangle$, where $1 \leq s \leq n-1$ for $p\geq 3$, and $1 \leq s < n-1$ for $p=2$.
		\end{enumerate}
	\end{lemma}
	\begin{proof}
		Let $p$ be a prime. Consider an ordinary, finite, non-cyclic metacyclic $p$-group $G$ with the following unique reduced presentation:
		\begin{equation*}
			G = \langle a, b ~ : ~ a^{p^n} = 1, \, b^{p^m} = a^{p^{n-r}}, \, bab^{-1} = a^{1+p^{n-s}} \rangle,
		\end{equation*}
		for certain integers $m, n, r$ and $s$ (as defined in \eqref{prest:metacyclic}). Here, $Z(G) = \langle a^{p^s}, \, b^{p^s} \rangle$ and $|Z(G)| = p^{n+m-2s}$. Furthermore, observe that $n-r \geq s$. Since $G$ is a $p$-group, $G$ has a faithful irreducible complex character if and only if $Z(G)$ is cyclic (see \cite[Theorem 2.32]{I}). We proceed with the rest of the proof in the following two cases.
		\begin{enumerate}
			\item {\bf Case ($s < m$).} In this case, we have the following claim.\\
			{\bf Claim.} If $n-r > s$, then $Z(G)$ is not cyclic.\\
			{\bf Proof of the claim.} We have 
			\begin{align*}
				(a^{p^si}b^{p^sj})^{p^{n+m-2s-1}} &= (a^{p^n})^{ip^{m-s-1}} (b^{p^m})^{jp^{n-s-1}}\\
				&= (a^{p^{n-r}})^{jp^{n-s-1}}\\
				&= (a^{p^n})^{jp^{n-r-s-1}}\\
				&= 1
			\end{align*}
			for some $i, j \in \mathbb{Z}$. Thus, $Z(G)$ does not have an element of order $p^{n+m-2s}$, and hence $Z(G)$ is not cyclic. This completes the proof of the claim.\\
			Moreover, observe that $\langle a^{p^s} \rangle \cap \langle b^{p^s} \rangle = \langle a^{p^{n-r}} \rangle = \langle a^{p^{s}} \rangle$ when $n-r = s$. Hence, $Z(G) = \langle b^{p^s} \rangle$ is cyclic when $n-r = s$. Therefore, in this case, $G$ is a faithful ordinary metacyclic $p$-group if and only if it is isomorphic to $G_1$.
			
			\item {\bf Case ($s = m$).} In this case, $Z(G)$ is necessarily cyclic because $\langle a^{p^s} \rangle \cap \langle b^{p^s} \rangle = \langle b^{p^m} \rangle = \langle b^{p^s} \rangle$, and hence $Z(G) = \langle a^{p^s} \rangle = \langle b^{p^s} \rangle$. Furthermore, $Z(G) = \langle a^{p^s} \rangle = \langle b^{p^s} \rangle = \langle b^{p^m} \rangle = \langle a^{p^{n-r}} \rangle$. Hence, $n-r = s$. Therefore, in this case, $G$ is a faithful ordinary metacyclic $p$-group if and only if it is isomorphic to $G_2$.
		\end{enumerate}
		This completes the proof of Lemma \ref{lemma:faithfulmeta}.
	\end{proof}

Under the set up of Lemma \ref{lemma:faithfulmeta}, we have the following Corollary. 	
	
	\begin{corollary}
		Let $G$ be a faithful ordinary metacyclic $p$-group, and let $\chi \in \Irr(G)$. Then $\chi \in \FIrr(G)$ if and only if $\chi(1)=p^s$.
	\end{corollary}
	\begin{proof}
		By Lemma \ref{lemma:faithfulmeta}, $G$ is isomorphic to $G_1$ or $G_2$. Furthermore, $\cd(G) = \{p^t : 0 \leq t \leq s\}$. Hence, by \cite[Proposition 3]{Mann}, the result follows.
	\end{proof}
	
Consider the groups mentioned in Lemma \ref{lemma:faithfulmeta}. Let $\chi \in \FIrr(G_1)$. Then, $\chi(1) = p^s$, and $\chi$ is defined as follows:
	\begin{equation}\label{eq:charactervalue G_1}
		\chi(a^i b^j) = \begin{cases}
			p^s \zeta_{p^{n+m-2s}}^{j_1} \zeta_{p^n}^i & \text{if } p^s \mid i \text{ and } j = j_1 p^s, \\
			0 & \text{otherwise,}
		\end{cases}
	\end{equation}
	where $\zeta_{p^{n+m-2s}}^{p^{m-s}} = \zeta_{p^n}^{p^s} = \zeta_{p^{n-s}}$. Here, $\zeta_{p^{n+m-2s}}$ denotes a primitive $p^{n+m-2s}$-th root of unity, and $\zeta_{p^n}$ denotes a primitive $p^n$-th root of unity (see \eqref{eq:chatractervalue}).
	
	Similarly, let $\chi \in \FIrr(G_2)$. Then, $\chi(1) = p^s$, and $\chi$ is defined as follows:
	\begin{equation}\label{eq::charactervalue G_2}
		\chi(a^i b^j) = \begin{cases}
			p^s \zeta_{p^{n-s}}^{i_1} & \text{if } i= i_1p^s \text{ and } j=0, \\
			0 & \text{otherwise,}
		\end{cases}
	\end{equation}
	where $\zeta_{p^{n-s}}$ denotes a primitive $p^{n-s}$-th root of unity (see \eqref{eq:chatractervalue}).
	
Now, we are ready to describe a method to find out rational matrix representations of an ordinary metacyclic $p$-group and prove Theorem \ref{thm:rationalrep}. We know that, for a finite group $G$ and $\chi \in \Irr(G)$, there exists a unique irreducible $\mathbb{Q}$-representation $\rho$ of $G$ such that $\chi$ appears as an irreducible constituent of $\rho \otimes_{\mathbb{Q}}\mathbb{F}$ with multiplicity $m_{\mathbb{Q}}(\chi)$, where $\mathbb{F}$ is a splitting field of $G$. For a linear complex character $\psi$ of $G$, the irreducible $\mathbb{Q}$-representation of $G$ with character $\Omega(\psi)$ can be derived using Lemma \ref{lemma:YamadaLinear}.
	
	\begin{lemma} \cite[Proposition 1]{Y}\label{lemma:YamadaLinear}
		Let $G$ be a finite group, and let $\psi \in \lin(G)$ and $N = \ker(\psi)$ with $n = [G : N]$. Suppose $G = \cup_{i = 0}^{n-1}Ny^i$. Then
		\[\psi(xy^i) = \zeta_n^i,~  (0 \leq i < n;~ x \in N),\]
		where $\zeta_n$ is a primitive $n$-th root of unity. 
		Now, let $f(X) = X^s - a_{s-1}X^{s-1} - \cdots -a_1X - a_0$ be the irreducible polynomial over $\mathbb{Q}$ such that $f(\zeta_n) = 0$, where $s = \phi(n)$ and
		\[\Psi(xy^i) = \left(\begin{array}{ccccc}
			0 & 1 & 0 & \cdots & 0\\
			0 & 0 & 1 & \cdots & 0\\
			\vdots & \vdots & \vdots & \ddots \\
			0 & 0 & \cdots & 0 & 1\\
			a_0 & a_1 & \cdots & \cdots & a_{s-1}
		\end{array}\right)^i,~ (0 \leq i < n;~ x \in N).\]
		Then $\Psi$ is an irreducible $\mathbb{Q}$-representation of $G$, whose character is $\Omega(\psi)$.
	\end{lemma}	 
Let $G$ be a finite group and $H$ a subgroup of $G$. Lemma \ref{lemma:Yamada} establishes the relationship between $\Omega(\psi)$ and $\Omega(\psi^G)$ for $\psi \in \Irr(H)$.

	\begin{lemma} \cite[Proposition 3]{Y}\label{lemma:Yamada}
		Let $G$ be a finite group, and let $H$ be a subgroup of $G$ and $\psi \in \Irr(H)$ such that $\psi^G \in\Irr(G)$. Then $m_\mathbb{Q}(\psi^G)$ divides $m_\mathbb{Q}(\psi)[\mathbb{Q}(\psi) : \mathbb{Q}(\psi^G)]$. Furthermore, the induced character $\Omega(\psi)^G$ of $G$ is a character of an irreducible $\mathbb{Q}$-representation of $G$, if and only if
		\[m_\mathbb{Q}(\psi^G) = m_\mathbb{Q}(\psi)[\mathbb{Q}(\psi) : \mathbb{Q}(\psi^G)].\]
		In this case, $\Omega(\psi)^G = \Omega(\psi^G)$.
	\end{lemma}
	
	Next, let $G$ be an ordinary metacyclic $p$-group, and consider $\chi \in \Irr(G)$. Since $G$ is monomial, there exists a subgroup $H \leq G$ and a character $\psi \in \lin(H)$ such that $\chi = \psi^G$. Additionally, by Lemma \ref{lemma:schurindex}, we know that $m_{\mathbb{Q}}(\chi) = 1$. Furthermore, by Lemma \ref{lemma:Yamada}, we see that $\Omega(\psi)^G = \Omega(\psi^G) = \Omega(\chi)$ if and only if $\mathbb{Q}(\psi) = \mathbb{Q}(\chi)$. 
	For a faithful ordinary metacyclic $p$-group $G$ and $\chi \in \FIrr(G)$, Theorem \ref{thm:reqpairmeta} provides a pair $(H, \psi)$, where $H$ is a subgroup of $G$, and $\psi \in \lin(H)$, such that $\psi^G = \chi$ and $\mathbb{Q}(\psi) = \mathbb{Q}(\chi)$.
	
	\begin{theorem} \label{thm:reqpairmeta}
		Let $p$ be a prime. Consider a faithful ordinary, finite, non-cyclic metacyclic $p$-group $G$. Let $\chi \in \FIrr(G)$. Then there exists an abelian subgroup $H$ of $G$ with $\psi \in \lin(H)$ such that $\psi^G = \chi$ and $\mathbb{Q}(\psi) = \mathbb{Q}(\chi)$. Moreover, we have the following.
		\begin{enumerate}
			\item Let $G=G_1$ and $\chi \in \FIrr(G)$ (as defined in \eqref{eq:charactervalue G_1}). Then we have the following two cases.
			\begin{enumerate}
				\item {\bf Case ($m \geq 2s$).} In this case, $H=\langle a, \, b^{p^s} \rangle$ and $\psi(x) = \begin{cases}
					\zeta_{p^n} & \text{if } x=a, \\
					\zeta_{p^{n+m-2s}} & \text{if } x=b^{p^s}.
				\end{cases}$
				\item {\bf Case ($m < 2s$).} In this case, $H=H_\mu=\langle a^{p^{2s-m}}, \, a^\mu b^{p^{m-s}} \rangle$ and $\psi \in \lin(H_\mu)$ is defined as follows:
				$$\psi(x) = \begin{cases}
					\zeta_{p^{n+m-2s}} & \text{if } x=a^{p^{2s-m}}, \\
					1 & \text{if } x=a^\mu b^{p^{m-s}},
				\end{cases}$$
				where $\mu$ is an integer such that $\mu k \equiv -1 \pmod{p^{n-s}}$, $-\mu l \equiv 1 \pmod{p^{n+m-2s}}$, $k=p^{-s}\frac{(1+p^{n-s})^{p^m}-1}{(1+p^{n-s})^{p^{m-s}}-1}$ and $l=p^{m-2s}\frac{(1+p^{n-s})^{p^s}-1}{(1+p^{n-s})^{p^{m-s}}-1}$.
			\end{enumerate}
			\item Let $G=G_2$ and $\chi \in \FIrr(G)$ (as defined in\eqref{eq::charactervalue G_2}). Then $H= \langle a^{p^s}, b \rangle$ and $\psi(x) = \begin{cases}
				\zeta_{p^{n-s}} & \text{if } x=a^{p^s}, \\
				1 & \text{if } x=b.
			\end{cases}$
		\end{enumerate}
	\end{theorem}
	\begin{proof}
		\begin{enumerate}
			\item 
			\begin{enumerate}
				\item {\bf Case ($m \geq 2s$).} In this case, we have $\mathbb{Q}(\psi)=\mathbb{Q}(\zeta_{p^{n+m-2s}})=\mathbb{Q}(\chi)$. Furthermore, $[G : H]=p^s$ and $G =\bigcup_{i=0}^{p^s-1}H b^i$. Hence, for $g \in G$, $\psi^G(g)=\sum_{i=0}^{p^s-1}{{\psi}^{\circ}}(b^igb^{-i})$, where ${\psi}^{\circ}$ is defined by  $\psi^\circ(x)=\psi(x)$ if $x \in H$ and  $\psi^{\circ}(x)=0$ if $x\notin H$. By using Lemma \ref{lemm:sumprimitveroots} and Lemma \ref{lemm:sumprimitveroots2}, we get $\psi^G=\chi$. This completes the proof of Theorem \ref{thm:reqpairmeta}(1)(a).
				
				\item {\bf Case ($m < 2s$).} In this case, we have
				\begin{align*}
					(a^\mu b^{p^{m-s}})^{p^s}&=a^{\mu\left(1+(1+p^{n-s})^{p^{m-s}} + (1+p^{n-s})^{2p^{m-s}} + \dots + (1+p^{n-s})^{(p^s-1)p^{m-s}}\right)}b^{p^m}\\
					&=a^{\mu\frac{(1+p^{n-s})^{p^m}-1}{(1+p^{n-s})^{p^{m-s}}-1}}a^{p^s}\\
					&=a^{\mu p^sk}a^{p^s}\\
					&=a^{p^s(\mu k +1)}\\
					&=1.
				\end{align*}
				Thus, observe that the order of $a^\mu b^{p^{m-s}}$ is $p^s$. Further, if $g \in \langle a^{p^{2s-m}} \rangle \cap \langle a^\mu b^{p^{m-s}} \rangle$, there exist $\alpha, \beta \in \mathbb{Z}$ such that $g = (a^{p^{2s-m}})^\alpha = (a^\mu b^{p^{m-s}})^\beta$. This implies that $b^{p^{m-s}\beta} \in \langle a \rangle$ and hence $p^s \mid \beta$. Therefore, $H_\mu=\langle a^{p^{2s-m}}, \, a^\mu b^{p^{m-s}} \rangle = \langle a^{p^{2s-m}} \rangle \times \langle a^\mu b^{p^{m-s}} \rangle$. Hence, $H_\mu$ is an abelian subgroup of $G$ and $\psi \in \lin(H_\mu)$ is well defined.\\
				Now, we have $(a^{p^{2s-m}})^{-\mu l} (a^\mu b^{p^{m-s}})^{p^{2s-m}}=b^{p^s}$. Thus, $b^{p^s}\in H_\mu$. Moreover, $b^i \notin H_\mu$ for $1 \leq i < p^s$. Observe that $[G : H_\mu]=p^s$. Therefore, $G =\bigcup_{i=0}^{p^s-1}H_\mu b^i$. Hence, for $g \in G$, $\psi^G(g)=\sum_{i=0}^{p^s-1}{{\psi}^{\circ}}(b^igb^{-i})$, where ${\psi}^{\circ}$ is defined by  $\psi^\circ(x)=\psi(x)$ if $x \in H_\mu$ and  $\psi^{\circ}(x)=0$ if $x\notin H_\mu$. Additionally, we have $\psi^G(a^{p^s})=p^s\zeta_{p^{n+m-2s}}^{p^{m-s}}=p^s\zeta_{p^{n-s}}=\chi(a^{p^s})$ and 
				$b^{p^s}=(a^{p^{2s-m}})^{-\mu l} (a^\mu b^{p^{m-s}})^{p^{2s-m}}$ which implies that $\psi^G(b^{p^s})=p^s\zeta_{p^{n+m-2s}}^{-\mu l}=p^s\zeta_{p^{n+m-2s}}=\chi(b^{p^s})$. More precisely, we can check that $\psi^G=\chi$. Furthermore, in this case, we have $\mathbb{Q}(\psi)=\mathbb{Q}(\zeta_{p^{n+m-2s}})=\mathbb{Q}(\chi)$. This completes the proof of Theorem \ref{thm:reqpairmeta}(1)(b).
			\end{enumerate}
			
			\item Here, we have $\mathbb{Q}(\psi)=\mathbb{Q}(\zeta_{p^{n-s}})=\mathbb{Q}(\chi)$. Furthermore, $[G : H]=p^s$ and $G =\bigcup_{i=0}^{p^s-1}H a^i$. Hence, for $g \in G$, $\psi^G(g)=\sum_{i=0}^{p^s-1}{{\psi}^{\circ}}(a^iga^{-i})$, where ${\psi}^{\circ}$ is defined by  $\psi^\circ(x)=\psi(x)$ if $x \in H$ and  $\psi^{\circ}(x)=0$ if $x\notin H$. It is routine to check that $\psi^G=\chi$. This completes the proof of Theorem \ref{thm:reqpairmeta}(2). \qedhere
		\end{enumerate} 
	\end{proof}
	
	Finally, consider a faithful ordinary metacyclic $p$-group $G$ as outlined in Lemma \ref{lemma:faithfulmeta}, and let $\chi \in \FIrr(G)$. According to Theorem \ref{thm:reqpairmeta}, there exists a subgroup $H$ of $G$ and $\psi \in \lin(H)$ such that $\psi^G = \chi$ and $\mathbb{Q}(\psi) = \mathbb{Q}(\chi)$. Consequently, Lemma \ref{lemma:Yamada} ensures that $\Omega(\chi) = \Omega(\psi)^G$. Additionally, by applying Lemma \ref{lemma:YamadaLinear}, we can determine an irreducible matrix representation $\Psi$ of $H$ over $\mathbb{Q}$ that affords the character $\Omega(\psi)$. The representation $\Psi^G$ then becomes an irreducible $\mathbb{Q}$-representation of $G$ that affords the character $\Omega(\psi)^G = \Omega(\psi^G) = \Omega(\chi)$. This leads us to Remark \ref{remark:requiredpair}.
	
	\begin{remark}\label{remark:requiredpair}
		\textnormal{To find an irreducible rational matrix representation of a faithful ordinary metacyclic $p$-group $G$ (for any prime $p$) that affords the character $\Omega(\chi)$, where $\chi \in \FIrr(G)$, it suffices to identify a pair $(H, \psi)$, where $H$ is a subgroup of $G$ and $\psi \in \lin(H)$, such that $\psi^G = \chi$ and $\mathbb{Q}(\psi) = \mathbb{Q}(\chi)$. This pair $(H, \psi)$ is referred to as a required pair for finding an irreducible rational matrix representation of $G$ that affords the character $\Omega(\chi)$. Reader can see \cite[Section 3]{Ram} for more details on required pairs.} 
	\end{remark}
	
	We conclude this section with an example. In Example \ref{example:rationalrepmeta}, we illustrate the process of finding an irreducible rational matrix representation of a faithful ordinary metacyclic $p$-group, using a required pair.
	
	\begin{example}\label{example:rationalrepmeta}
		\textnormal{Consider the group $G = \langle a^{32} = 1, \, b^{16} = a^8, \, bab^{-1} = a^5 \rangle$. The group $G$ is a faithful, ordinary metacyclic $2$-group with order $512$. Moreover, $G$ is of type $G_1$ as described in Lemma \ref{lemma:faithfulmeta} with $n = 5$, $m = 4$ and $s = 3$. Let $\chi \in \FIrr(G)$. In this case, $\chi(1) = 8$ and $\chi$ is defined as follows:
			\begin{equation*}
				\chi(a^i b^j) = \begin{cases}
					8\zeta_4^{i_1}\zeta_8^{j_1} & \text{if } i = 4i_1 \text{ and } j = 4j_1, \\
					0 & \text{otherwise},
				\end{cases}
			\end{equation*}
			where $\zeta_8^2 = \zeta_4$, $\zeta_4$ is a primitive $4$-th root of unity and $\zeta_8$ is a primitive $8$-th root of unity (refer to \eqref{eq:charactervalue G_1}). Since $m < 2s$, Theorem \ref{thm:reqpairmeta}(1)(b) implies that $(H_\mu, \psi)$ is the required pair for computing an irreducible rational matrix representation of $G$ corresponding to the character $\Omega(\chi)$.\\
			Next, we compute $\mu$. From $\mu k \equiv -1 \pmod{4}$, where $k = \frac{1}{8}\left(\frac{5^{16}-1}{5^2-1}\right) = 794728597$, and $-\mu l \equiv 1 \pmod{8}$, where $l = \frac{1}{4}\left(\frac{5^8-1}{5^2-1}\right) = 4069$, we obtain $\mu = 3$. Thus, we set $\mu = 3$, and hence $H_\mu = H_3 = \langle a^4, a^{3}b^2 \rangle$ with $\psi \in \lin(H_\mu)$ such that $\psi(a^4) = \zeta_8$ and $\psi(a^{3}b^2) = 1$. Furthermore, the kernel of $\psi$ is $\ker(\psi) = \langle a^{3}b^2 \rangle = N$ (say). Since $[H_\mu : N] = 8$ and $H_\mu = \bigcup_{i=0}^{7} Na^{4i}$, by Lemma \ref{lemma:YamadaLinear}, the map $\Psi : H_\mu \rightarrow GL_4(\mathbb{Q})$ provides an irreducible rational matrix representation of degree $4$ for $H_\mu$ that affords the character $\Omega(\psi)$. The explicit form of $\Psi$ is given by
			\begin{equation*}
				\Psi(a^4)= \left(\begin{array}{cccc}
					0 & 1 & 0 & 0\\
					0 & 0 & 1 & 0\\
					0 & 0 & 0 & 1\\
					-1 & 0 & 0 & 0
				\end{array}\right), \, \, \textnormal{and} \, \,
				\Psi(a^{3}b^2)= \left(\begin{array}{cccc}
					1 & 0 & 0 & 0\\
					0 & 1 & 0 & 0\\
					0 & 0 & 1 & 0\\
					0 & 0 & 0 & 1
				\end{array}\right) = I_4.
			\end{equation*}
			Define $\Psi(a^4) = P$ and let $O$ represent the zero matrix of order $4$. Since $[G : H_\mu] = 8$, we have $G = \bigcup_{i=0}^{7} H_\mu b^i$. Consequently, the induced representation $\Psi^G$ is an irreducible rational matrix representation of degree $32$ of $G$ affording the character $\Omega(\chi)$. The explicit form of $\Psi^G : G \rightarrow GL_{32}(\mathbb{Q})$ is given by
			\begin{equation*}
				\Psi^G(a)= \left(\begin{array}{cccccccc}
					O & O & P^7 & O & O & O & O & O \\
					O & O & O & I_4 & O & O & O & O \\
					O & O & O & O & P^5 & O & O & O \\
					O & O & O & O & O & P^6 & O & O \\
					O & O & O & O & O & O & P^3 & O \\
					O & O & O & O & O & O & O & P^4 \\
					P^2 & O & O & O & O & O & O & O \\
					O & P^3 & O & O & O & O & O & O 
				\end{array}\right), \, \, \textnormal{and} \, \,
				\Psi^G(b)= \left(\begin{array}{cccccccc}
					O & I_4 & O & O & O & O & O & O \\
					O & O & I_4 & O & O & O & O & O \\
					O & O & O & I_4 & O & O & O & O \\
					O & O & O & O & I_4 & O & O & O \\
					O & O & O & O & O & I_4 & O & O \\
					O & O & O & O & O & O & I_4 & O \\
					O & O & O & O & O & O & O & I_4 \\
					P & O & O & O & O & O & O & O 
				\end{array}\right).
		\end{equation*}}
	\end{example}

		\section{Rational group algebra of ordinary metacyclic $p$-groups}\label{sec:mainresult}
  Here, we present the proof of Theorem \ref{thm:wedderburnmetacyclic}.
	\begin{proof}[Proof of Theorem \ref{thm:wedderburnmetacyclic}.]
		Let $p$ be a prime and let $\zeta_d$ be a primitive $d$-th root of unity for some positive integer $d$. Consider an ordinary, finite, non-cyclic metacyclic $p$-group $G$ with a unique reduced presentation:
		\begin{equation*}
			G = \langle a, b ~ : ~ a^{p^n} = 1, \, b^{p^m} = a^{p^{n-r}}, \, bab^{-1} = a^{1+p^{n-s}} \rangle,
		\end{equation*}
		for certain integers $m, n, r$, and $s$ (as described in \eqref{prest:metacyclic}). Let $\chi \in \Irr(G)$, and assume that $\rho$ is an irreducible $\mathbb{Q}$-representation of $G$ affording the character $\Omega(\chi)$. Denote by $A_\mathbb{Q}(\chi)$ the simple component of the Wedderburn decomposition of $\mathbb{Q}G$ corresponding to $\rho$, which is isomorphic to $M_q(D)$ for some $q \in \mathbb{N}$ and division ring $D$. By Lemma \ref{lemma:schurindex}, we have $m_\mathbb{Q}(\chi) = 1$. Furthermore, since $[D : Z(D)] = m_\mathbb{Q}(\chi)^2$ and $Z(D) = \mathbb{Q}(\chi)$ (see \cite[Theorem 3]{IR}), it follows that $D = Z(D) = \mathbb{Q}(\chi)$. Next, consider $\rho = \bigoplus_{i=1}^l \rho_i$, where $l = [\mathbb{Q}(\chi) : \mathbb{Q}]$ and each $\rho_i$ is an irreducible complex representation of $G$ affording the character $\chi^{\sigma_i}$ for some $\sigma_i \in \operatorname{Gal}(\mathbb{Q}(\chi)/\mathbb{Q})$. Since $m_\mathbb{Q}(\chi) = 1$, we have $q = \chi(1)$ (see \cite[Theorem 3.3.1]{JR}).
		
		\noindent Now, let $\chi \in \lin(G)$. Suppose $\rho$ is the irreducible $\mathbb{Q}$-representation of $G$ that affords the character $\Omega(\chi)$. Let $\bar{\chi} \in \operatorname{Irr}(G/G')$ be such that $\bar{\chi}(gG') = \chi(g)$. Then $A_\mathbb{Q}(\chi) \cong \mathbb{Q}(\bar{\chi})$. Moreover, $|\lin(G)| = |\Irr(G/G')|$, and $G/G' = \langle a^{p^s}, b \rangle \cong C_{p^{n-s}} \times C_{p^m}$. Therefore, from Remark \ref{remark:linear_metacyclic}, we have the following two cases.
		
		\noindent {\bf Case A ($n-s \geq m$).} In this case, the simple components of the Wedderburn decomposition of $\mathbb{Q}G$ corresponding to all irreducible $\mathbb{Q}$-representations of $G$ whose kernels contain $G'$ are
			$$\mathbb{Q} \bigoplus_{\lambda=1}^m (p^\lambda+p^{\lambda-1})\mathbb{Q}(\zeta_{p^\lambda}) \bigoplus_{\lambda=m+1}^{n-s}p^m \mathbb{Q}(\zeta_{p^\lambda})$$
		in $\mathbb{Q}G$.
		
		\noindent {\bf Case B ($n-s < m$).} In this case, the simple components of the Wedderburn decomposition of $\mathbb{Q}G$ corresponding to all irreducible $\mathbb{Q}$-representations of $G$ whose kernels contain $G'$ are
		$$\mathbb{Q} \bigoplus_{\lambda=1}^{n-s} (p^\lambda+p^{\lambda-1})\mathbb{Q}(\zeta_{p^\lambda}) \bigoplus_{\lambda=n-s+1}^{m}p^{n-s} \mathbb{Q}(\zeta_{p^\lambda})$$
		in $\mathbb{Q}G$.\\
		
		\noindent Next, let $\chi \in \nl(G)$ and suppose $\rho$ is the irreducible $\mathbb{Q}$-representation of $G$ affording the character $\Omega(\chi)$. Here, $A_\mathbb{Q}(\chi) \cong M_q(D)$ and $D = \mathbb{Q}(\chi)$. Since $\operatorname{cd}(G) = \{p^t : 0 \leq t \leq s\}$, $q = \chi(1) = p^t$ for some $1 \leq t \leq s$. The rest of the proof proceeds through the following cases:
		\begin{enumerate}
			\item {\bf Case ($n-s \geq m$).} In this case, the group $G$ is always split, which implies that $r=0$ (refer to \eqref{prest:metacyclic} for details). Additionally, for all $t$ with $1 \leq t \leq s$, we have $n-s > m-t$. Consequently, according to Sub-case 1 of Case 1 in Remark \ref{remark:galoisclasses_metacyclic}, there are $p^{m-t}$ distinct Galois conjugacy classes of size $\phi(p^{n-s})$ for a fixed $t$ (where $1 \leq t \leq s$). Therefore, for each fixed $t$ (where $1 \leq t \leq s$), we have $\mathbb{Q}(\chi) = \mathbb{Q}(\zeta_{p^{n-s}})$, and there are $p^{m-t}$ distinct irreducible $\mathbb{Q}$-representations of $G$ that afford the character $\Omega(\chi)$, with $\chi \in \Irr(G)$ and $\chi(1) = p^t$. Thus, for $n-s \geq m$, the simple components of the Wedderburn decomposition of $\mathbb{Q}G$ corresponding to all irreducible $\mathbb{Q}$-representations of $G$ whose kernels do not contain $G'$ are given by
		        $$\bigoplus_{t=1}^{s} p^{m-t}M_{p^t}(\mathbb{Q}(\zeta_{p^{n-s}})).$$
			This expression, combined with the results from Case A, completes the proof of Theorem \ref{thm:wedderburnmetacyclic}(1).
			
			\item \textbf{Case ($n-s < m$).} Assume $m = (n-s) + k$. We consider the following two sub-cases:
			
			\begin{enumerate}
				\item \textbf{Sub-case ($k \leq s-r$).} Here, $k \leq s - r$, which implies $m - n + s \leq s - r$, leading to $r \leq n - m$. In this sub-case, $G$ is always split, meaning $r = 0$ (see \eqref{prest:metacyclic} for details). Furthermore, in this scenario, $n-s \geq m-t$ for $k \leq t \leq s$. Therefore, from Sub-case 1 of Case 1 in Remark \ref{remark:galoisclasses_metacyclic}, there are $p^{m-t}$ distinct Galois conjugacy classes of size $\phi(p^{n-s})$ for a fixed $t$ (where $k \leq t \leq s$). Consequently, for each fixed $t$ (where $t \in \{k, k+1, \dots, s\}$), $\mathbb{Q}(\chi) = \mathbb{Q}(\zeta_{p^{n-s}})$, and there are $p^{m-t}$ distinct irreducible $\mathbb{Q}$-representations of $G$ that afford the character $\Omega(\chi)$, with $\chi \in \Irr(G)$ and $\chi(1) = p^t$. Thus, the simple components of the Wedderburn decomposition of $\mathbb{Q}G$ corresponding to all irreducible $\mathbb{Q}$-representations of $G$ that afford the character $\Omega(\chi)$ with $\chi(1) = p^t$ (for $k \leq t \leq s$) are given by
				$$\bigoplus_{t=k}^{s} p^{m-t}M_{p^t}(\mathbb{Q}(\zeta_{p^{n-s}})).$$
				
				Additionally, in this sub-case, $n-s < m-t$ for $1 \leq t \leq k-1$. Thus, from Sub-case 2 of Case 1 in Remark \ref{remark:galoisclasses_metacyclic}, there are $p^{n-s}$ distinct Galois conjugacy classes of size $\phi(p^{n-s})$ and $\phi(p^{n-s})$ distinct Galois conjugacy classes of size $\phi(p^\lambda)$ for $n-s < \lambda \leq m-t$ for a fixed $t$ (where $1 \leq t \leq k-1$). Consequently, for each fixed $t$ (where $t \in \{1, 2, \dots, k-1\}$), there are $p^{n-s}$ distinct irreducible $\mathbb{Q}$-representations of $G$ that afford the character $\Omega(\chi)$ with $\chi \in \Irr(G)$, $\chi(1) = p^t$ and $\mathbb{Q}(\chi) = \mathbb{Q}(\zeta_{p^{n-s}})$, as well as $\phi(p^{n-s})$ distinct irreducible $\mathbb{Q}$-representations of $G$ with $\chi \in \Irr(G)$, $\chi(1) = p^t$ and $\mathbb{Q}(\chi) = \mathbb{Q}(\zeta_{p^\lambda})$ for $n-s < \lambda \leq m-t$. Therefore, the simple components of the Wedderburn decomposition of $\mathbb{Q}G$ corresponding to all irreducible $\mathbb{Q}$-representations of $G$ with $\chi(1) = p^t$ (for $1 \leq t \leq k-1$) are given by
				$$\bigoplus_{t=1}^{k-1} p^{n-s}M_{p^t}(\mathbb{Q}(\zeta_{p^{n-s}})) \bigoplus_{t=1}^{k-1}\bigoplus_{\lambda=n-s+1}^{m-t} (p^{n-s}-p^{n-s-1})M_{p^t}(\mathbb{Q}(\zeta_{p^{\lambda}})).$$
				
				Including the above expressions from this sub-case along with Case B completes the proof of Theorem \ref{thm:wedderburnmetacyclic}$(2)(a)$.
				
				\item \textbf{Sub-case ($k > s-r$).} In this sub-case, $n-s < m-t$ for $1 \leq t \leq s-r$. Thus, from Sub-case 2 of Case 1 in Remark \ref{remark:galoisclasses_metacyclic}, there are $p^{n-s}$ distinct Galois conjugacy classes of size $\phi(p^{n-s})$ and $\phi(p^{n-s})$ distinct Galois conjugacy classes of size $\phi(p^\lambda)$ for $n-s < \lambda \leq m-t$ for a fixed $t$ (where $1 \leq t \leq s-r$). Therefore, for each fixed $t$ (where $t \in \{1, 2, \dots, s-r\}$), there are $p^{n-s}$ distinct irreducible $\mathbb{Q}$-representations of $G$ that afford the character $\Omega(\chi)$ with $\chi \in \Irr(G)$, $\chi(1) = p^t$ and $\mathbb{Q}(\chi) = \mathbb{Q}(\zeta_{p^{n-s}})$, and $\phi(p^{n-s})$ distinct irreducible $\mathbb{Q}$-representations of $G$ with $\chi \in \Irr(G)$, $\chi(1) = p^t$ and $\mathbb{Q}(\chi) = \mathbb{Q}(\zeta_{p^\lambda})$ for $n-s < \lambda \leq m-t$. Hence, the simple components of the Wedderburn decomposition of $\mathbb{Q}G$ corresponding to all irreducible $\mathbb{Q}$-representations of $G$ with $\chi(1) = p^t$ (for $1 \leq t \leq s-r$) are given by
				$$\bigoplus_{t=1}^{s-r} p^{n-s}M_{p^t}(\mathbb{Q}(\zeta_{p^{n-s}})) \bigoplus_{t=1}^{s-r}\bigoplus_{\lambda=n-s+1}^{m-t} (p^{n-s}-p^{n-s-1})M_{p^t}(\mathbb{Q}(\zeta_{p^{\lambda}})).$$
				
				Additionally, according to Case 2 of Remark \ref{remark:galoisclasses_metacyclic}, there are $p^{n - r - t}$ distinct Galois conjugacy classes of irreducible complex characters of degree $p^t$, each with size $\phi(p^{m + r - s})$ for a fixed $t$ where $s - r + 1 \leq t \leq s$. Thus, for each fixed $t$ (where $t \in \{s - r + 1, \dots, s\}$), $\mathbb{Q}(\chi) = \mathbb{Q}(\zeta_{p^{m + r - s}})$, and there are distinct irreducible $\mathbb{Q}$-representations of $G$ that afford the character $\Omega(\chi)$ with $\chi \in \Irr(G)$ and $\chi(1) = p^t$. Therefore, the simple components of the Wedderburn decomposition of $\mathbb{Q}G$ corresponding to all irreducible $\mathbb{Q}$-representations of $G$ with $\chi(1) = p^t$ (for $s-r+1 \leq t \leq s$) are given by
				$$\bigoplus_{t=s-r+1}^s p^{n-r-t}M_{p^t}(\mathbb{Q}(\zeta_{p^{m+r-s}})).$$
				
				Including the above expressions from this sub-case along with Case B completes the proof of Theorem \ref{thm:wedderburnmetacyclic}$(2)(b)$.
			\end{enumerate}
			
		\end{enumerate}
		This completes the proof of Theorem \ref{thm:wedderburnmetacyclic}.
	\end{proof}
	
	\begin{corollary}
		Let $p$ be a prime and let $\zeta_d$ be a primitive $d$-th root of unity for some positive integer $d$. Consider a non-split ordinary, finite, non-cyclic metacyclic $p$-group $G$, with a unique reduced presentation:
		\begin{equation*}
			G = \langle a, b ~ : ~ a^{p^n} = 1, \, b^{p^m} = a^{p^{n-r}}, \, bab^{-1} = a^{1+p^{n-s}} \rangle,
		\end{equation*}
		for certain integers $m, n, r$ and $s$ (as defined in \eqref{prest:metacyclic}). Then
		\begin{align*}
			\mathbb{Q}G \cong & \mathbb{Q} \bigoplus_{\lambda=1}^{n-s} (p^\lambda+p^{\lambda-1})\mathbb{Q}(\zeta_{p^\lambda}) \bigoplus_{\lambda=n-s+1}^{m}p^{n-s} \mathbb{Q}(\zeta_{p^\lambda}) \bigoplus_{t=1}^{s-r} p^{n-s}M_{p^t}(\mathbb{Q}(\zeta_{p^{n-s}}))\\ &\bigoplus_{t=1}^{s-r}\bigoplus_{\lambda=n-s+1}^{m-t} (p^{n-s}-p^{n-s-1})M_{p^t}(\mathbb{Q}(\zeta_{p^{\lambda}})) \bigoplus_{t=s-r+1}^s p^{n-r-t}M_{p^t}(\mathbb{Q}(\zeta_{p^{m+r-s}})).
		\end{align*}
	\end{corollary}
	\begin{proof}
		Since $G$ is non-split, we have $n-m+1 <s$. This implies that $m >n-s$ (see \eqref{prest:metacyclic}). Further, $n-m+1 \leq r$, leading to $m-n+s > s-r$ (see \eqref{prest:metacyclic}). Hence, the proof follows from Theorem \ref{thm:wedderburnmetacyclic}(2)(b).
	\end{proof}
	
	\begin{corollary}\label{cor:isogroupalgebra}
		Let $G$ and $H$ be two ordinary metacyclic $p$-groups, where $p$ is a prime. Then $\mathbb{Q}G \cong \mathbb{Q}H$ if and only if $G \cong H$.
	\end{corollary}
	\begin{proof}
		The proof directly follows from Theorem \ref{thm:wedderburnmetacyclic}.
	\end{proof}
	
	\subsection{Examples}\label{exam:metacyclic} In \cite{Ram2}, a combinatorial formula was derived to compute the Wedderburn decomposition of the rational group algebra for a split metacyclic $p$-group, where $p$ is an odd prime. Theorem \ref{thm:wedderburnmetacyclic} extends the result of \cite[Theorem 1]{Ram2}. In this sub-section, we illustrate Theorem \ref{thm:wedderburnmetacyclic} by providing  Wedderburn decomposition of some ordinary metacyclic $p$-groups.
	
	\begin{enumerate}
		\item Consider $G_1 = \langle a, b : a^{32} = 1, b^8 = 1, bab^{-1} = a^9 \rangle$. In this case, $n=5$, $m=3$, $r=0$, and $s=2$. Thus, by applying Theorem \ref{thm:wedderburnmetacyclic}(1), we get
		$$\mathbb{Q}G_1 \cong 4\mathbb{Q} \oplus 6\mathbb{Q}(\zeta_4) \oplus 12\mathbb{Q}(\zeta_8) \oplus 4M_2(\mathbb{Q}(\zeta_8)) \oplus 2M_4(\mathbb{Q}(\zeta_8)).$$
		In the GAP library, $G_1$ is isomorphic to SmallGroup$(256, 318)$.
		
		\item Consider $G_2 = \langle a, b : a^{16} = 1, b^{16} = 1, bab^{-1} = a^5 \rangle$. Here, $n=4$, $m=4$, $r=0$, and $s=2$. By Theorem \ref{thm:wedderburnmetacyclic}(2)(a), we get
		\begin{align*}
			\mathbb{Q}G_2 \cong & 4\mathbb{Q} \oplus 6\mathbb{Q}(\zeta_4) \oplus 4\mathbb{Q}(\zeta_8) \oplus 4\mathbb{Q}(\zeta_{16}) \oplus 4M_2(\mathbb{Q}(\zeta_4)) \oplus 2M_2(\mathbb{Q}(\zeta_8)) \\
			& \oplus 4M_4(\mathbb{Q}(\zeta_4)).
		\end{align*}
		$G_2$ is isomorphic to SmallGroup$(256, 41)$ in the GAP library.
		
		\item Consider $G_3 = \langle a, b : a^8 = 1, b^{32} = 1, bab^{-1} = a^5 \rangle$. Here, $n=3$, $m=5$, $r=0$, and $s=1$. Using Theorem \ref{thm:wedderburnmetacyclic}(2)(b), we get
		\begin{align*}
			\mathbb{Q}G_3 \cong & 4\mathbb{Q} \oplus 6\mathbb{Q}(\zeta_4) \oplus 4\mathbb{Q}(\zeta_8) \oplus 4\mathbb{Q}(\zeta_{16}) \oplus 4\mathbb{Q}(\zeta_{32}) \oplus 4M_2(\mathbb{Q}(\zeta_4)) \oplus 2M_2(\mathbb{Q}(\zeta_8)) \\
			& \oplus 2M_2(\mathbb{Q}(\zeta_{16})).
		\end{align*}
		$G_3$ is isomorphic to SmallGroup$(256, 319)$ in GAP library.
		
		\item Consider $G_4 = \langle a, b : a^{16} = 1, b^{16} = a^8, bab^{-1} = a^5 \rangle$. For this group, $n=4$, $m=4$, $r=1$, and $s=2$. From Theorem \ref{thm:wedderburnmetacyclic}(2)(b), we get
		\begin{align*}
			\mathbb{Q}G_4 \cong & 4\mathbb{Q} \oplus 6\mathbb{Q}(\zeta_4) \oplus 4\mathbb{Q}(\zeta_8) \oplus 4\mathbb{Q}(\zeta_{16}) \oplus 4M_2(\mathbb{Q}(\zeta_4)) \oplus 2M_2(\mathbb{Q}(\zeta_8)) \\
			& \oplus 2M_4(\mathbb{Q}(\zeta_8)).
		\end{align*}
		In GAP library, $G_4$ is isomorphic to SmallGroup$(256, 320)$.
		
		\item Finally, consider $G_5 = \langle a, b : a^{27} = 1, b^{27} = a^9, bab^{-1} = a^4 \rangle$. Here, $n=3$, $m=3$, $r=1$, and $s=2$. Using Theorem \ref{thm:wedderburnmetacyclic}(2)(b), we get
		\begin{align*}
			\mathbb{Q}G_5 \cong & \mathbb{Q} \oplus 4\mathbb{Q}(\zeta_3) \oplus 3\mathbb{Q}(\zeta_9) \oplus 3\mathbb{Q}(\zeta_{27}) \oplus 3M_3(\mathbb{Q}(\zeta_3)) \oplus 2M_3(\mathbb{Q}(\zeta_9)) \\
			& \oplus M_9(\mathbb{Q}(\zeta_9)).
		\end{align*}
		In the GAP library, $G_5$ is isomorphic to SmallGroup$(729, 92)$.
	\end{enumerate}

	\section{Acknowledgements}
	Ram Karan acknowledges University Grants Commission, Government of India.

\end{document}